\documentclass[11pt]{article}

\usepackage[a4paper]{geometry}
\usepackage{amssymb, amsmath, amsfonts}

\usepackage{graphicx}
\usepackage[caption = false]{subfig}
\usepackage{enumerate}
\usepackage{comment}
\usepackage{etoolbox,setspace}

\usepackage{xr}
\usepackage{prettyref}
\newrefformat{Seq}{(\ref{#1})}
\renewcommand{\theequation}{\arabic{equation}}

\usepackage{natbib, float}

\usepackage{amsthm}
\binoppenalty=\maxdimen
\relpenalty=\maxdimen

\allowdisplaybreaks

\newcommand{\ds}{\displaystyle}
\newcommand{\E}{\text{E}}

\newcommand{\X}{\mathsf{X}}

\newcommand{\real}{{\mathbb R}}

\newcommand\numberthis{\addtocounter{equation}{1}\tag{\theequation}}

\theoremstyle{plain}
\newtheorem{theorem}{Theorem}

\newtheorem{lemma}{Lemma}
\newtheorem{corollary}{Corollary}

\theoremstyle{remark}
\newtheorem{remark}{Remark}
\newtheorem{cond}{Condition}

\newtheorem{example}{Example}

 \author{ Dootika Vats\\ Department of Statistics\\ University of
  Warwick\\ \texttt{D.Vats@warwick.ac.uk} \and James M. Flegal
  \thanks{Research supported by the National Science
    Foundation.}\\ Department of Statistics\\ University of
  California, Riverside\\ \texttt{jflegal@ucr.edu} \and Galin L. Jones
  \thanks{Research supported by the National Institutes of Health and
    the National Science Foundation.}\\ School of
  Statistics\\ University of Minnesota\\ \texttt{galin@umn.edu} }
\title{Multivariate Output Analysis for Markov Chain Monte Carlo} \date{\today}

\begin{document}

\maketitle

\doublespacing
\begin{abstract}
Markov chain Monte Carlo (MCMC) produces a correlated sample for estimating expectations with respect to a target distribution. A fundamental question is when should sampling stop so that we have good estimates of the desired quantities? The key to answering this question lies in assessing the Monte Carlo error through a multivariate Markov chain central limit theorem (CLT).  The multivariate nature of this Monte Carlo error largely has been ignored in the MCMC literature. We present a multivariate framework for terminating simulation in MCMC. We define a multivariate effective sample size, estimating which requires strongly consistent estimators of the covariance matrix in the Markov chain CLT; a property we show for the multivariate batch means estimator. We then provide a lower bound on the number of minimum effective samples required for a desired level of precision. This lower bound depends on the problem only in the dimension of the expectation being estimated, and not on the underlying stochastic process. This result is obtained by drawing a connection between terminating simulation via effective sample size and terminating simulation using a relative standard deviation fixed-volume sequential stopping rule; which we demonstrate is an asymptotically valid procedure. The finite sample properties of the proposed method are demonstrated in a variety of examples.
\end{abstract}

\section{Introduction} 
\label{sec:introduction}

Markov chain Monte Carlo (MCMC) algorithms are used to estimate expectations with respect to a probability distribution when independent sampling is difficult. Typically, interest is in estimating a vector of quantities. However, analysis of MCMC output routinely focuses on inference about complicated joint distributions only through their marginals.  This, despite the fact that the assumption of independence across components holds only rarely in settings where MCMC is relevant. Thus standard univariate convergence diagnostics, sequential stopping rules for termination, effective sample size definitions, and confidence intervals all lead to an incomplete understanding of the estimation process.  We overcome the drawbacks of univariate analysis by developing a methodological framework for multivariate analysis of MCMC output.

Let $F$ be a distribution with support $\mathcal{X}$ and $g: \mathcal{X} \to \mathbb{R}^{p}$ be an $F$-integrable function such that $\theta := \E_F g$ is of interest. If $\{X_t\}$ is an $F$-invariant Harris recurrent Markov chain, set $\{Y_t\}=\{g(X_t)\}$ and  estimate $\theta$ with $\theta_n = n^{-1} \sum_{t=1}^{n} Y_t$ since $\theta_n \to \theta$, with probability 1, as $n \to \infty$.  Finite sampling leads to an unknown \textit{Monte Carlo error}, $\theta_n - \theta$, estimating which is essential to assessing the quality of estimation. If for $\delta > 0$, $g$ has $2 + \delta$ moments under $F$ and $\{X_t\}$ is polynomially ergodic of order $m > (2 + \delta)/\delta$, an approximate sampling distribution for the Monte Carlo error is available via a Markov chain central limit theorem (CLT). That is, there exists a $p \times p$ positive definite matrix, $\Sigma$,  such that as $n \to \infty$, 
\begin{equation} \label{eq:multi_clt}
  \sqrt{n}(\theta_n - \theta) \overset{d}{\to} N_p(0, \Sigma) \; .
\end{equation}

Thus the CLT describes asymptotic behavior of the Monte Carlo error and the strong law for $\theta_n$ ensures that large $n$ leads to a small Monte Carlo error. But, how large is large enough? This question has not been adequately addressed in the literature since current approaches are based on the univariate CLT 
\begin{equation} \label{eq:uni_clt}
  \sqrt{n}(\theta_{n,i} - \theta_{i}) \overset{d}{\to} N(0, \sigma_i^2) \; \text{ as } n \to \infty,
\end{equation}
where $\theta_{n,i}$ and $\theta_i$ are the $i$th components of $\theta_n$ and $\theta$ respectively and $\sigma_i^2$ is the $i$th diagonal element of $\Sigma$. Notice that a univariate approach ignores cross-correlation across components, leading to an inaccurate understanding of the estimation process.

Many output analysis tools that rely on \eqref{eq:uni_clt} have been developed for MCMC (see \cite{atch:2011}, \cite{atch:2016}, \cite{fleg:jone:2010}, \cite{fleg:gong:2015}, \cite{gelm:rubi:1992a}, \cite{gong:fleg:2015}, and \cite{jone:hara:caff:neat:2006}). To determine termination, \cite{jone:hara:caff:neat:2006} implemented the \textit{fixed-width sequential stopping rule} where simulation is terminated the first time the width of the confidence interval for each component is small. More formally, for a desired tolerance of $\epsilon_i$ for component $i$, the rule terminates simulation the first time after some $n^* \ge 0$ iterations, for all components
\begin{equation*} \label{eq:absolute fixed}
t_{*} \dfrac{\sigma_{n,i}}{\sqrt{n}}  +n^{-1} \le \epsilon_i,
\end{equation*}
where $\sigma^{2}_{n,i}$ is a strongly consistent estimator of $\sigma_i^2$, and $t_{*}$ is an appropriate $t$-distribution quantile. The role of $n^*$ is to ensure a minimum simulation effort (as defined by the user) so as to avoid poor initial estimates of $\sigma_i^2$. This rule laid the foundation for termination based on quality of estimation rather than convergence of the Markov chain. As a consequence, estimation is reliable in the sense that if the procedure is repeated again, the estimates will not be vastly different \citep{fleg:hara:jone:2008}. However, implementing the fixed-width sequential stopping rule can be challenging since (a) careful analysis is required for choosing $\epsilon_i$ for each $\theta_{n,i}$ which can be tedious or even impossible for large $p$; (b) to ensure the right coverage probability, $t_*$ is chosen to account for multiple confidence intervals (often by using a Bonferroni correction). Thus when $p$ is even moderately large, these termination rules can be aggressively conservative leading to delayed termination; (c) simulation stops when each component satisfies the termination criterion; therefore, all cross-correlations are ignored and termination is governed by the slowest mixing components; and (d) it ignores correlation in the target distribution.

To overcome the drawbacks of the fixed-width sequential stopping rule, we propose the \textit{relative standard deviation fixed-volume sequential stopping rule} that differs from the \cite{jone:hara:caff:neat:2006} procedure in two fundamental ways; (a) it is motivated by the multivariate CLT in \eqref{eq:multi_clt} and not by the univariate CLT in \eqref{eq:uni_clt}; and (b)  it terminates simulation not by the absolute size of the confidence region, but by its size relative to the inherent variability in the problem.  Specifically, simulation stops when the Monte Carlo standard error is small compared to the variability in the target distribution. Naturally, an estimate of the Monte Carlo standard error is required and for now, we assume that $\Sigma$ can be estimated consistently. Later we will discuss procedures for estimating $\Sigma$. The relative standard deviation fixed-volume sequential stopping rule terminates the first time after some user-specified $n^* \ge 0$  iterations
\begin{equation}
\label{eq:intro_rule}
\text{Volume of Confidence Region}^{1/p} + n^{-1} < \epsilon |\Lambda_n|^{1/2p} \; ,
\end{equation}
where $\Lambda_n$ is the sample covariance matrix, $| \cdot |$ denotes determinant, and $\epsilon$ is the tolerance level. As in the univariate setting, the role of $n^*$ is to avoid premature termination due to early bad estimates of $\Sigma$ or $\Lambda$; we will say more about how to choose $n^*$ in Section~\ref{sec:multivariate_effective_sample_size}. 

\cite{wilks:1932} defines the determinant of a covariance matrix as the \textit{generalized variance}. Thus, an equivalent interpretation of \eqref{eq:intro_rule} is that simulation is terminated when the generalized variance of the Monte Carlo error is small relative to the generalized variance of $g$ with respect to $F$; that is, a scaled estimate of $|\Sigma|$ is small compared to the estimate of $|\Lambda| = |\text{Var}_{F}(Y_1)|$. We call $|\Lambda|^{1/2p}$ the \textit{relative metric}.  For $p = 1$, our choice of the relative metric reduces \eqref{eq:intro_rule} to the relative standard deviation fixed-width sequential stopping rule of \cite{fleg:gong:2015}. 

We show that if the estimator for $\Sigma$ is strongly consistent, the stopping rule in \eqref{eq:intro_rule} is asymptotically valid, in that the confidence regions created at termination have the right coverage probability as $\epsilon \to 0$. Our result of asymptotic validity holds for a wide variety of relative metrics. A different choice of the relative metric, leads to a fundamentally different approach to termination. For example, if instead of choosing $|\Lambda|^{1/2p}$ as the relative metric, we choose a positive constant, then our work provides a multivariate generalization of the absolute-precision procedure considered by \cite{jone:hara:caff:neat:2006}. 

Another standard way of terminating simulation is to stop when the number of effective samples for each component reaches a pre-specified lower bound (see \cite{atk:gray:drum:2008}, \cite{drum:ho:phill:ramb:2006}, \cite{gior:brod:jord:2015},  \cite{gong:fleg:2015}, and \cite{krus:2014} for a few examples).  We focus on a multivariate study of effective sample size (ESS) since univariate treatment of ESS ignores cross-correlations across components, thus painting an inaccurate picture of the quality of the sample. To the best of our knowledge, a multivariate approach to ESS has not been studied in the literature. We define 
\[ \text{ESS} = n \left(\dfrac{|\Lambda|}{|\Sigma|} \right)^{1/p}.\]
When there is no correlation in the Markov chain, $\Sigma = \Lambda$ and ESS $ = n$. Notice that our definition of ESS involves the ratio of generalized variances. This ratio also occurs in \eqref{eq:intro_rule} which helps us arrive at a key result; terminating according to the relative standard deviation fixed-volume sequential stopping rule is asymptotically equivalent to terminating when the estimated ESS satisfies
\[\widehat{\text{ESS}} \geq W_{p, \alpha, \epsilon}, \]
where $W_{p, \alpha, \epsilon}$ can be calculated \textit{a priori} and is a function only of the dimension of the estimation problem, the level of confidence of the confidence regions, and the relative precision desired. Thus, not only do we show that terminating via ESS is a valid procedure, we also provide a theoretically valid, practical lower bound on the number of effective samples required. 

Recall that we require a strongly consistent estimator of $\Sigma$. Estimating $\Sigma$ is a difficult problem due to the serial correlation in the Markov chain. 
\cite{vats:fleg:jones:2017} demonstrated strong consistency for a class of multivariate spectral variance estimators while \cite{dai:jone:2017} introduced multivariate initial sequence estimators and established their asymptotic validity. However, both estimators are expensive to calculate and do not scale well with either $p$ or $n$. Instead, we use the  \textit{multivariate batch means} (mBM) estimator of $\Sigma$ which is significantly faster to compute (see Section \ref{sec:discussion}) and requires weaker moment conditions on $g$ for strong consistency. Our strong consistency result weakens the conditions required in \cite{jone:hara:caff:neat:2006} for the univariate batch means (uBM) estimator. In particular, we do not require a one-step minorization and only require polynomial ergodicity (as opposed to geometric ergodicity). The condition is fairly weak since often the existence of the Markov chain CLT itself is demonstrated via polynomial ergodicity or a stronger result (see \cite{jone:2004} for a review). Many Markov chains used in practice have been shown to be at least polynomially ergodic. See \cite{acos:hube:jone:2015}, \cite{doss:hobe:2010}, \cite{hobe:geye:1998}, \cite{jarn:robe:2002}, \cite{jarn:hans:2000}, \cite{jarn:robe:2002}, \cite{john:jone:neat:2013}, \cite{john:jone:2015}, \cite{jone:robe:rose:2014}, \cite{jone:hobe:2004}, \cite{khare:hobe:2013}, \cite{marc:hobe:2004} \cite{robe:pols:1994}, \cite{tan:jone:hobe:2013}, \cite{tan:hobe:2012}, \cite{vats:2016}, among many others.

The  multivariate stopping rules terminate earlier than univariate methods since (a) termination is dictated by the joint behavior of the components of the Markov chain and not by the component that mixes the slowest (b) using the inherent multivariate nature of the problem and acknowledging cross-correlations leads to a more realistic understanding of the estimation process, and (c) avoiding corrections for multiple testing give considerably smaller confidence regions even in moderate $p$ problems. There are also cases where univariate methods \emph{cannot} be implemented due to large memory requirements. On the other hand, the multivariate methods are inexpensive relative to the sampling time for the Markov chain and terminate significantly earlier. We present one such example in Section \ref{sec:spatio_example} through a Bayesian dynamic spatial-temporal model.

The rest of the paper is organized as follows. In the sequel we present a motivating Bayesian logistic regression model. In Section \ref{sec:termination_rules} we formally introduce a general class of relative fixed-volume sequential termination rules. In Section \ref{sec:multivariate_effective_sample_size} we define ESS and provide a lower bound on the number of effective samples required for simulation. Our theoretical results in these sections require a strongly consistent estimator for $\Sigma$; a result we show for the mBM estimator in Section \ref{sec:multivariate_batch_means}. In Section \ref{sec:examples} we continue our implementation of the Bayesian logistic regression model and consider additional examples. We choose a vector autoregressive process of order 1, where the convergence rate of the process can be manipulated. Specifically, we construct the process in such a way that one component mixes slowly, while the others are fairly well behaved. Such behavior is often seen in hierarchical models with priors on the variance components. The next example is 
that of a Bayesian lasso where the posterior is in 51 dimensions. We also implement our output analysis methods for a fairly complicated Bayesian dynamic spatial temporal model. For this example we do not know if our assumptions on the process hold, thus demonstrating the situation users often find themselves in. We conclude with a discussion in Section \ref{sec:discussion}. 

\subsection{An Illustrative Example} 
\label{sub:illust_example}

For $i = 1, \dots, K$, let $Y_i$ be a binary response variable and $X_i = (x_{i1}, x_{i2}, \dots, x_{i5})$ be the observed predictors for the $i$th observation. Assume $\tau^2$ is known,

\begin{equation}\label{eq:logistic model}
Y_i | X_i, \beta  \overset{ind}{\sim} \text{Bernoulli} \left( \dfrac{1}{1 + e^{-X_i \beta}}  \right)\, ,~~~\text{ and } ~~~\beta  \sim N_5(0, \tau^2 I_5)\; .
\end{equation}
This simple hierarchical model results in an intractable posterior, $F$ on $\mathbb{R}^5$. The dataset used is the \texttt{logit} dataset in the \texttt{mcmc} R package. The goal is to estimate the posterior mean of $\beta$, $\E_F\beta$. Thus $g$ here is the identity function mapping to $\mathbb{R}^5$. We implement a random walk Metropolis-Hastings algorithm with a multivariate normal proposal distribution $N_5( \;\cdot \;, 0.35^2 I_5)$ where $I_5$ is the $5 \times 5$ identity matrix and the $0.35$ scaling approximates the optimal acceptance probability suggested by \cite{rob:gel:gilks}.


We calculate the Monte Carlo estimate for $\E_F \beta$ from an MCMC sample of size $10^5$. The starting value for $\beta$ is a random draw from the prior distribution. We use the mBM estimator described in Section \ref{sec:multivariate_batch_means} to estimate $\Sigma$.  We also implement the uBM methods described in \cite{jone:hara:caff:neat:2006} to estimate $\sigma_i^2$, which captures the autocorrelation in each component while ignoring the cross-correlation.  This cross-correlation is often significant as seen in Figure \ref{fig:blog_trace_acf}, and can only be captured by multivariate methods like mBM. In Figure \ref{fig:ellipse_conf_region} we present $90\%$ confidence regions created using mBM and uBM estimators for $\beta_1$ and $\beta_3$ (for the purpose of this figure, we set $p = 2$). This figure illustrates why multivariate methods are likely to outperform univariate methods. The confidence ellipse is the smallest volume region for a particular level of confidence. Thus, these confidence ellipses are likely to be preferred over other confidence regions.  
\begin{figure}[h]
\begin{center}
 \subfloat[]{
\includegraphics[width = 2.5in]{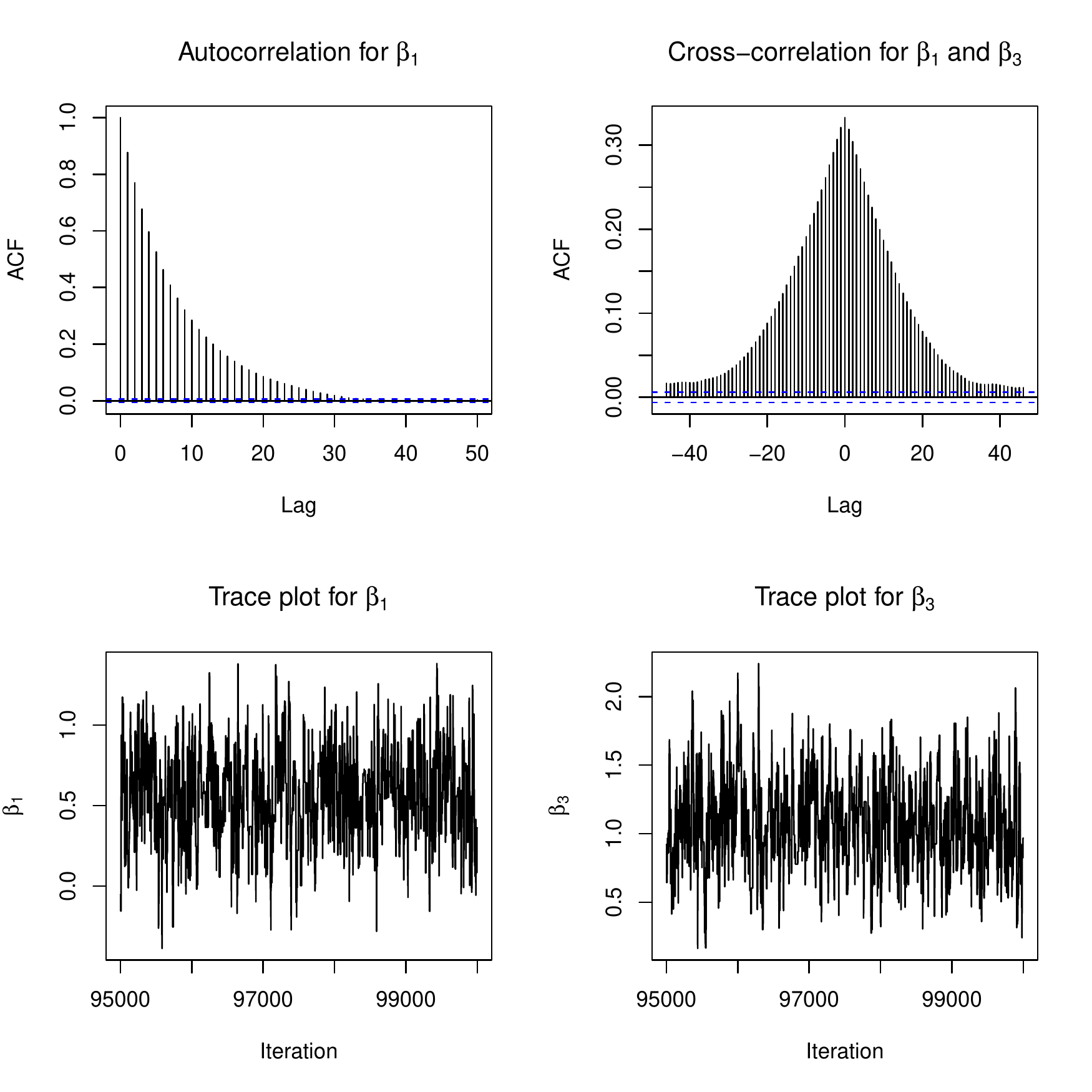} \label{fig:blog_trace_acf}}
 \subfloat[]{
 \includegraphics[width = 2.5in]{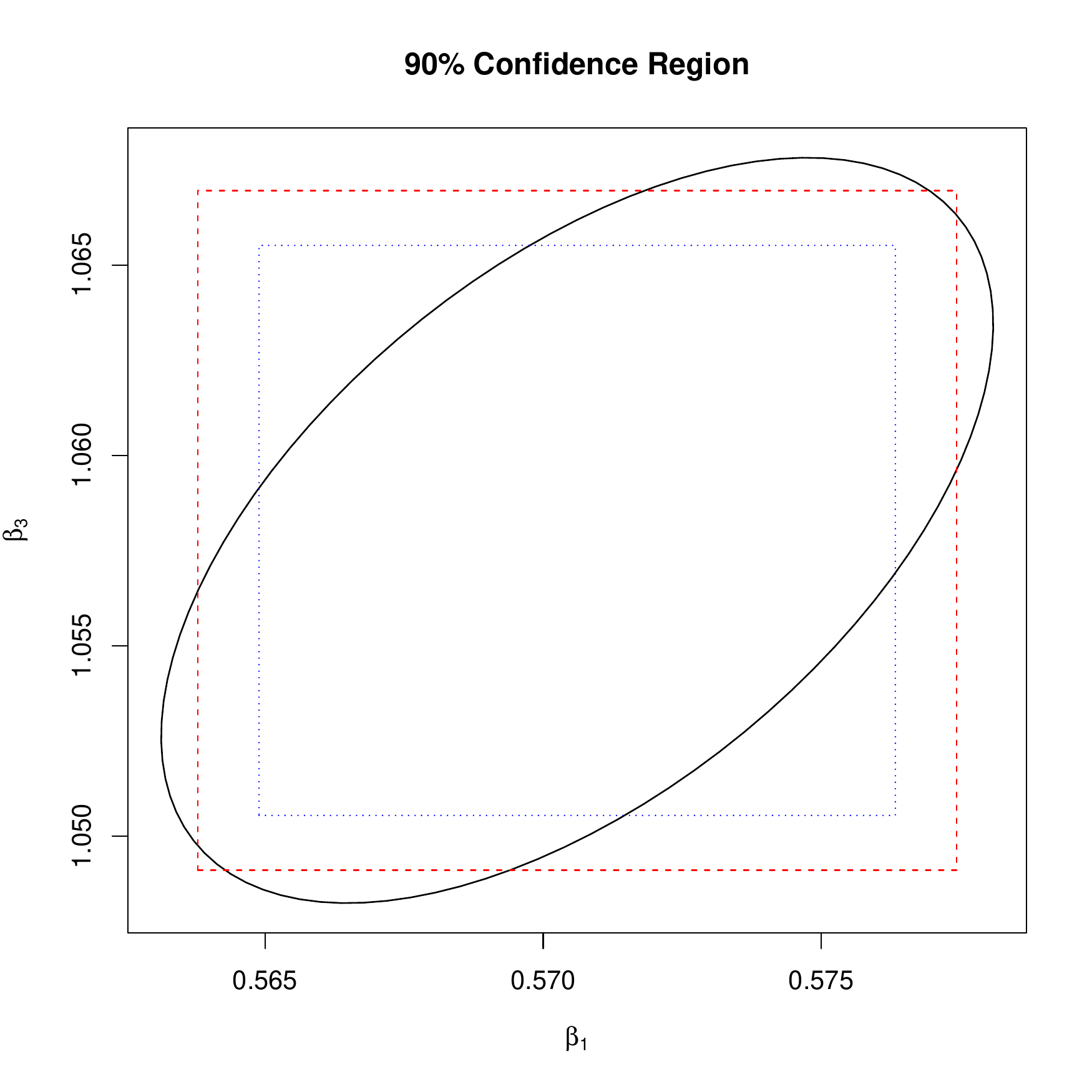} \label{fig:ellipse_conf_region}} 
 \caption{\footnotesize(a) ACF plot for $\beta_1$, cross-correlation plot between $\beta_1$ and $\beta_3$, and trace plots for $\beta_1$ and $\beta_3$. (b) Joint 90\% confidence region for $\beta_1$ and $\beta_3$. The ellipse is made using mBM, the dotted line using uncorrected uBM, and the dashed line using the uBM corrected by Bonferroni. The Monte Carlo sample size is $10^5$ for both plots.} 
\end{center}
 \end{figure}


To assess the confidence regions, we verify their coverage
probabilities over 1000 independent replications with Monte Carlo
sample sizes in $\{10^4, 10^5, 10^6\}$. Since the true posterior mean is unknown, we use $(0.5706, 0.7516, 1.0559, 0.4517, 0.6545)$ obtained by averaging over $10^9$ iterations as a proxy. For each of the 1000 replications, it was
noted whether the confidence region contained the true posterior mean. The volume of the confidence region to the $p$th
root was also observed. Table \ref{tab:blog_coverage} summarizes the
results. Note that though the uncorrected univariate methods produce
the smallest confidence regions, their coverage probabilities are far
from desirable. For a large enough Monte Carlo sample size, mBM
produces 90\% coverage probabilities with systematically lower
volume than uBM corrected with Bonferroni (uBM-Bonferroni).
\begin{table}[h]
\footnotesize 
  \caption{ \footnotesize \label{tab:blog_coverage}Volume to the $p$th ($p=5$) root and coverage probabilities for 90\% confidence regions constructed using mBM, uBM uncorrected, and uBM corrected by Bonferroni. Replications = 1000 and standard errors are indicated in parenthesis.}
\begin{center}
  \begin{tabular}{c|ccc}
  \hline
  $n$& mBM & uBM-Bonferroni & uBM \\
  \hline
  \multicolumn{4}{c}{Volume to the $p$th root} \\ 

  \hline
  1e4 & 0.062 \tiny{(7.94e-05)} &  0.066 \tiny{(9.23e-05)} & 0.046 \tiny{(6.48e-05)}\\
  1e5 & 0.020 \tiny{(1.20e-05)} &  0.021 \tiny{(1.42e-05)} & 0.015 \tiny{(1.00e-05)}\\
  1e6 & 0.006 \tiny{(1.70e-06)} &  0.007 \tiny{(2.30e-06)} & 0.005 \tiny{(1.60e-06)}\\
  \hline
  \multicolumn{4}{c}{Coverage Probabilities} \\ 
  \hline
  1e4 & 0.876 \tiny{(0.0104)} &  0.889 \tiny{(0.0099)} & 0.596 \tiny{(0.0155)}\\ 
  1e5 & 0.880 \tiny{(0.0103)} &  0.910 \tiny{(0.0090)} & 0.578 \tiny{(0.0156)}\\
  1e6 & 0.894 \tiny{(0.0097)} &  0.913 \tiny{(0.0094)} & 0.627 \tiny{(0.0153)}\\ \hline
  \end{tabular}
  \end{center}
\end{table}

Thus even simple MCMC problems produce complex
dependence structures within and across components of the
samples. Ignoring this structure leads to an incomplete
understanding of the estimation process. Not only do we gain more information about the Monte Carlo error using multivariate methods, but we also avoid using conservative Bonferroni methods.

\section{Termination Rules} 
\label{sec:termination_rules}

We consider multivariate sequential termination rules that lead to asymptotically valid confidence regions. Let $T^2_{1-\alpha, p, q} $ denote the $1-\alpha$ quantile  of a Hotelling's T-squared distribution with dimensionality parameter $p$ and degrees of freedom $q$. Throughout this section and the next, we assume $\Sigma_n$ is a strongly consistent estimator of $\Sigma$. A $100(1- \alpha) \%$ confidence region  for $\theta$ is the set
\[ 
C_\alpha(n) = \left\{ \theta \in \mathbb{R}^p: n(\theta_n - \theta)^T \Sigma^{-1}_{n} (\theta_n - \theta) < T^2_{1-\alpha, p, q} \right\}\, , 
\]
where $q$ is determined by the choice of $\Sigma_n$. Then $C_\alpha(n)$ forms an ellipsoid in $p$ dimensions oriented along the directions of the eigenvectors of $\Sigma_{n}$. The volume of $C_{\alpha}(n)$ is 
\begin{equation}
\label{eq:vol_conf}
\text{Vol} (C_\alpha(n)) = \dfrac{2\pi^{p/2} }{p \Gamma(p/2)} \left( \dfrac{T^2_{1-\alpha, p, q}}{n}  \right)^{p/2} | \Sigma_{n} |^{1/2}\, .
\end{equation}
Since $p$ is fixed and  $\Sigma_{n} \to \Sigma$ with probability 1, Vol($C_\alpha(n)$) $ \to 0$, with probability 1, as $n \to \infty$.  If $\epsilon > 0 $ and $s(n)$ is a positive real valued function defined on the positive integers, then a fixed-volume sequential stopping rule terminates the simulation at the random time 
\begin{equation}
\label{eq:glynn_whitt_rule}
  T(\epsilon) = \inf \left \{n \geq 0: \text{Vol}(C_\alpha(n))^{1/p} + s(n) \leq \epsilon   \right\}\, .
\end{equation}
\cite{glyn:whit:1992} provide conditions so that terminating at $T(\epsilon)$ yields
confidence regions that are asymptotically valid in that, as $\epsilon
\to 0, \Pr\left[\theta \in C_\alpha(T(\epsilon) ) \right] \to 1 -
\alpha$. In particular, they let $s(n)=\epsilon I(n < n^*) + n^{-1}$  which ensures simulation does not
terminate before $n^* \geq 0$ iterations. The sequential stopping rule \eqref{eq:glynn_whitt_rule} can be difficult to implement in practice since the choice of $\epsilon$ depends on the units of $\theta$, and has to be carefully chosen for every application.  We present an alternative to  \eqref{eq:glynn_whitt_rule}  which can be used more naturally and which we will show connects nicely to the idea of ESS.  

Let $\| \cdot \|$ denote the Euclidean norm. Let $K(Y,p)>0$ be an attribute of the estimation process and suppose $K_n(Y,p) > 0$ is an estimator of $K(Y,p)$; for example, take $\|\theta\| = K(Y,p)$ and $\|\theta_n\| = K_n(Y,p)$.  Set $s(n) = \epsilon K_n(Y,p)I(n < n^*) + n^{-1}$ and define 
\begin{equation*} \label{eq:universal_relative}
  T^*(\epsilon) = \inf \left\{n \geq 0: \text{Vol}(C_\alpha(n))^{1/p} + s(n) \leq \epsilon K_n(Y,p) \right\}.
\end{equation*}
We call $K(Y,p)$ the relative metric. The following result establishes asymptotic validity of this termination rule.  The proof is provided in the supplementary material.

\begin{theorem}
\label{thm:asymp_valid}
Let $g: \X \to \real^p$ be such that $\E_F\|g\|^{2 + \delta} < \infty$ for some $\delta > 0$ and let $X$ be an $F$-invariant polynomially ergodic Markov chain of order $m > (1+ \epsilon_1)(1+ 2/\delta)$ for some $\epsilon_1 > 0$. If $K_n(Y,p) \to K(Y,p)$  with probability 1 and $\Sigma_n \to \Sigma$ with probability 1, as $n \to \infty$, then, as  $\epsilon \to 0,\, T^*(\epsilon) \to \infty$ and $\Pr\left[\theta \in C_\alpha(T^*(\epsilon)) \right] \to 1 - \alpha$.
\end{theorem}   
\begin{remark}
Theorem \ref{thm:asymp_valid} applies when  $K(Y,p) = K_n(Y,p) = 1$. This choice of the relative metric leads to the absolute-precision fixed-volume sequential stopping rule; a multivariate generalization of the procedure considered by \cite{jone:hara:caff:neat:2006}.
\end{remark}

Suppose $K(Y,p)=|\Lambda|^{1/2p} =|\text{Var}_{F}Y_1|^{1/2p}$ and if $\Lambda_n$ is the usual sample covariance matrix for $\{Y_t\}$,  set  $K_n(Y,p)=|\Lambda_n|^{1/2p}$. Note that $\Lambda_n$ is positive definite as long as $n > p$, so $|\Lambda_n|^{1/2p} > 0$.  Then $K_n(Y, p) \to K(Y,p)$, with probability 1, as $n\to \infty$ and $T^*(\epsilon)$ is the first time the variability in estimation (measured via the volume of the confidence region) is an $\epsilon$th fraction of the variability in the target distribution. The \textit{relative standard deviation fixed-volume sequential stopping rule} is  formalized as terminating at random time
\begin{equation}
\label{eq:rel_sd_rule}
  T_{SD}(\epsilon) = \inf \left\{n \geq 0: \text{Vol}(C_\alpha(n))^{1/p} + \epsilon |\Lambda_n|^{1/2p} I(n < n^*) + n^{-1}\leq \epsilon |\Lambda_n|^{1/2p} \right\} \, .
\end{equation}


\section{Effective Sample Size} 
\label{sec:multivariate_effective_sample_size}

\cite{gong:fleg:2015}, \cite{kass:carlin:gelman:neal:1998}, \cite{liu:2008}, and \cite{robe:case:2013} define ESS for the $i$th component of the process as
\[
\text{ESS}_i = \dfrac{n}{1 + 2\sum_{k=1}^{\infty} \rho(Y^{(i)}_1, Y^{(i)}_{1+k})} = n \dfrac{\lambda_i^2}{\sigma_i^2}, 
\]
where $\rho(Y^{(i)}_1, Y^{(i)}_{1+k})$ is the lag $k$ correlation for the $i$th component of $Y$, $\sigma_i^2$ is the $i$th diagonal element of $\Sigma$, and $\lambda_i^2$ is the $i$th diagonal element of $\Lambda$. 
A strongly consistent estimator of ESS$_i$ is obtained through strongly consistent estimators of $\lambda_i^2$ and $\sigma_i^2$ via the sample variance $(\lambda^2_{n,i})$ and univariate batch means estimators $(\sigma^2_{n,i})$, respectively.
ESS$_i$ is then estimated for each component separately, and a conservative estimate of the overall ESS is taken to be the minimum of all ESS$_i$.  This leads to the situation where the estimate of ESS is dictated by the components that mix the slowest, while ignoring all other components.

Instead of using the diagonals of $\Lambda$ and $\Sigma$ to define ESS, we use the matrices themselves. Let $S_p^+$ denote the set of all $p \times p$ positive definite matrices. Univariate quantification of the matrices requires a mapping $S_p^+ \to \real_+$ that captures the variability described by the covariance matrix. We use the determinant since for a random vector, the determinant of its covariance matrix is its generalized variance. The concept of generalized variance was first introduced by \cite{wilks:1932} as a univariate measure of spread for a multivariate distribution. \cite{wilks:1932} recommended the use of the $p$th root of the generalized variance. This was formalized by \cite{sengupta:1987} as the \textit{standardized generalized variance} in order to compare variability over different dimensions. We define 
\begin{equation*}
\label{eq:multi_ESS}
  \text{ESS} = n \left( \dfrac{|\Lambda|}{|\Sigma|} \right)^{1/p}\, .
\end{equation*}
When $p = 1$, the ESS reduces to the form of univariate ESS presented above. Let $\Lambda_n$ be the sample covariance matrix of $\{Y_t\}$ and $\Sigma_n$ be a strongly consistent estimator of $\Sigma$. Then a strongly consistent estimator of ESS is
\begin{equation*}
\label{eq:multi_ESS_est}
  \widehat{\text{ESS}} = n \left(\dfrac{|\Lambda_n|}{|\Sigma_n|}\right)^{1/p}.
\end{equation*}

\subsection{Lower Bound for Effective Sample Size}
\label{subsec:ess}

Rearranging the defining inequality in \eqref{eq:rel_sd_rule} yields that when $n \ge n^*$ 
\begin{align*}
  \widehat{\text{ESS}} & \geq  \left[ \left(\dfrac{2 \pi^{p/2}}{p \Gamma(p/2)} \right)^{1/p} \left(T^2_{1-\alpha, p, q}\right)^{1/2}  + \frac{ |\Sigma_n|^{-1/2p}}{n^{1/2}}\right]^{2} \dfrac{1}{\epsilon^2}  \approx \dfrac{2^{2/p} \pi}{(p \Gamma(p/2))^{2/p}} \left(T^2_{1-\alpha, p, q}\right) \dfrac{1}{\epsilon^2}\, .  
\end{align*} 
Thus, the relative standard deviation fixed-volume sequential stopping rule is equivalent to terminating the first time $\widehat{\text{ESS}}$ is larger than a lower bound. This lower bound is a function of $n$ through $q$ and thus is difficult to determine before starting the simulation. However, as $n \to \infty$, the scaled $T^2_{p,q}$ distribution converges to a $\chi^2_p$, leading to the following approximation 
\begin{equation}
\label{eq:lower_bound_ess}
\widehat{\text{ESS}} \geq \dfrac{2^{2/p} \pi}{(p \Gamma(p/2))^{2/p}} \dfrac{\chi^2_{1-\alpha, p}}{\epsilon^2}\, .   
\end{equation}
One can \textit{a priori} determine the number of effective samples required for their choice of $\epsilon$ and $\alpha$. As $p \to \infty$, the lower bound in \eqref{eq:lower_bound_ess} converges to $2\pi e/\epsilon^2$. Thus for large $p$, the lower bound is mainly determined by the choice of $\epsilon$. 
On the other hand, for a fixed $\alpha$, having obtained $W$ effective samples, the user can use the lower bound to determine the relative precision ($\epsilon$) in their estimation. In this way, \eqref{eq:lower_bound_ess} can be used to make informed decisions regarding termination.

\begin{example} Suppose $p=5$ (as in the logistic regression setting of Section~\ref{sub:illust_example}) and that we want a precision of $\epsilon = .05$ (so the Monte Carlo error is 5\% of the variability in the target distribution) for a $95\%$ confidence region.  This requires $\widehat{\text{ESS}} \geq 8605$.  On the other hand, if we simulate until $\widehat{\text{ESS}} = 10000$, we obtain a precision of $\epsilon = .0464$.  \end{example}

\begin{remark}\label{rem:nstar} Let $n_{\text{pos}}$ be the smallest integer such
   that $\Sigma_{n_{\text{pos}}}$ is positive definite; in the next section we will discuss how to choose $n_{\text{pos}}$ for the mBM estimator. In light of the lower bound in \eqref{eq:lower_bound_ess}, a natural choice of $n^*$ is
 \begin{equation} \label{eq:nstar_bound}
  n^* \geq \max \left\{ n_{\text{pos}}, \dfrac{2^{2/p} \pi}{(p \Gamma(p/2))^{2/p}} \dfrac{\chi^2_{1-\alpha, p}}{\epsilon^2} \right\}\,. 
\end{equation}
\end{remark}

\section{Strong Consistency of Multivariate Batch Means Estimator} 
\label{sec:multivariate_batch_means}

In Sections \ref{sec:termination_rules} and \ref{sec:multivariate_effective_sample_size} we assumed the existence of a strongly consistent estimator of $\Sigma$. A class of multivariate spectral variance estimators were shown to be strongly consistent by \cite{vats:fleg:jones:2017}. However, when $p$ is large, this class of estimators is expensive to calculate as we show in Section \ref{sec:discussion}.  Thus, we present the relatively inexpensive mBM estimator and provide conditions for strong consistency.

Let $n = a_n b_n$, where $a_n$ is the number of batches and $b_n$ is the batch size. For $k = 0, \dots, a_n-1$, define $\bar{Y}_k := b_n^{-1} \sum_{t = 1}^{b_n} Y_{kb_n + t}$. Then $\bar{Y}_k$ is the mean vector for batch $k$ and the mBM estimator of $\Sigma$ is given by \begin{equation} \label{eq:mbm}
  \Sigma_{n} = \dfrac{b_n}{a_n - 1} \ds \sum_{k=0}^{a_n - 1} \left(\bar{Y}_k - \theta_n \right) \left( \bar{Y}_k - \theta_n  \right)^T.
\end{equation}
For the mBM estimator, $q$ in \eqref{eq:vol_conf} is $a_n - p$. In addition, $\Sigma_n$ is singular if $a_n < p$, thus $n_{\text{pos}}$ is the smallest $n$ such that $a_{n} > p$. 

When $Y_t$ is univariate, the batch means estimator has been well studied for MCMC problems \citep{jone:hara:caff:neat:2006,fleg:jone:2010} and for steady state simulations \citep{dame:1994,glyn:igle:1990,glyn:whit:1991}. \cite{glyn:whit:1991} showed that the batch means estimator cannot be consistent for fixed batch size, $b_n$. \cite{dame:1994,damerdji:1995}, \cite{jone:hara:caff:neat:2006} and \cite{fleg:jone:2010} established its asymptotic properties including strong consistency and mean square consistency when \textit{both} the batch size and number of batches increases with $n$.   

The multivariate extension as in \eqref{eq:mbm} was first introduced by \cite{chen:seila:1987}. For steady-state simulation, \cite{charnes:1995} and \cite{muno:glyn:2001} studied confidence regions for $\theta$ based on the mBM estimator, however, its asymptotic properties remain unexplored. In Theorem \ref{thm:mbm}, we present conditions for strong consistency of $\Sigma_n$ in estimating $\Sigma$ for MCMC, but our results hold for more general processes. Our main assumption on the process is that of a \textit{strong invariance principle} (SIP).

\begin{cond}
  \label{cond:msip}
  Let $\|\cdot\|$ denote the Euclidean norm and $ \{B(t), t\geq 0\}$ be a $p$-dimensional multivariate Brownian motion. There exists a $p \times p$ lower triangular matrix $L$, a nonnegative increasing function $\gamma$ on the positive integers, a finite random variable $D$, and a sufficiently rich probability space such that, with probability 1, as $n \to \infty$,
  \begin{equation}\label{eq:sip}
    \| n(\theta_n - \theta) - LB(n)\| < D \gamma(n)\,.
  \end{equation}
\end{cond}



\begin{cond}
\label{cond:bn_conditions}
The batch size $b_n$ satisfies the following conditions,
  \begin{enumerate}
    \item \label{cond:bn_to_n}
  the batch size $b_n$ is an integer sequence such that $b_n \to \infty$ and $n/b_n \to \infty$ as $n\to \infty$ where, $b_n$ and $n/b_n$ are monotonically increasing,
\item \label{cond:bn_by_n_c}
  there exists a constant $c \geq 1$ such that $\sum_{n} \left({b_n}{n}^{-1}\right)^c < \infty$.
  \end{enumerate}
\end{cond}
In Theorem \ref{thm:mbm} we show strong consistency of $\Sigma_n$. The proof is given in the supplementary material 






\begin{theorem}
\label{thm:mbm}
Let $g$ be such that $E_F \|g\|^{2 + \delta} < \infty$ for some $\delta > 0$. Let $X$ be an $F$-invariant polynomially ergodic Markov chain of order $m > (1 + \epsilon_1)(1+2/\delta)$ for some $\epsilon_1 > 0$.  Then \eqref{eq:sip} holds with $\gamma(n) = n^{1/2 - \lambda}$ for some $\lambda >0 $. If Condition \ref{cond:bn_conditions} holds and $b_n^{-1/2} (\log n)^{1/2} n^{1/2 - \lambda} \to 0$ as $n \to \infty$, then $\Sigma_{n} \to \Sigma$, with probability 1, as $n \to \infty$.
\end{theorem}

\begin{remark}\label{rem:proof_general}
The theorem holds more generally outside the context of Markov chains for processes that satisfy Condition \ref{cond:msip}. This includes independent processes \citep{berk:phil:1979,einm:1989,zait:1998}, Martingale sequences \citep{eber:1986}, renewal processes \citep{horv:1984} and $\phi$-mixing and strongly mixing processes \citep{kuel:phil:1980,dehl:phil:1982}. The general statement of the theorem is provided in the supplementary material
\end{remark}

\begin{remark}
\label{rem:lambda}
Using Theorem 4 from \cite{kuel:phil:1980}, \cite{vats:fleg:jones:2017} established Condition 1 with $\gamma(n) = n^{1/2 - \lambda}$, for some $\lambda > 0$ for polynomially ergodic Markov chains.  We use their result directly. \cite{kuel:phil:1980} show that $\lambda$ only depends on $p$, $\epsilon$ and $\delta$, however the exact relationship remains an open problem. 
For slow mixing processes $\lambda$ is closer to $0$ while for fast mixing processes $\lambda$ is closer to $1/2$ \citep{dame:1991, dame:1994}.
\end{remark}

\begin{remark}\label{rem:requiring_reg}
It is natural to consider $b_n = \lfloor n^{\nu} \rfloor$ for $ 0 < \nu < 1 $. Then $\lambda$ in the SIP must satisfy $\lambda > (1 - \nu)/4$ so that $b_n^{-1/2} (\log n)^{1/2} n^{1/2 - \lambda} \to 0$ as $n \to \infty$. 
Since $\nu > 1 - 2\lambda$,  smaller batch sizes suffice for fast mixing processes and slow mixing processes require larger batch sizes. This reinforces our intuition that higher correlation calls for larger batch sizes. Calibrating $\nu$ in $b_n = \lfloor n^{\nu} \rfloor$ is essential to ensuring the mBM estimates perform well in finite samples.  Using mean square consistency of univariate batch means estimators, \cite{fleg:jone:2010} concluded that an asymptotically optimal batch size is proportional to $\lfloor n^{1/3} \rfloor$. 
\end{remark}

\begin{remark}\label{rem:weaker_conditions}
For $p = 1$, \cite{jone:hara:caff:neat:2006} proved strong consistency of the batch means estimator under the stronger assumption of geometric ergodicity and a one-step minorization, which we do not make. Thus, in Theorem \ref{thm:mbm} while extending the result of strong consistency to $p \geq 1$, we also weaken the conditions for the univariate case.
\end{remark}


\begin{remark}\label{rem:eigen_con}
By Theorem 3 in \cite{vats:fleg:jones:2017}, strong consistency of the mBM estimator  implies strong consistency of its eigenvalues. 
\end{remark}

\section{Examples} 
\label{sec:examples}

In each of the following examples we present a target distribution $F$, a Markov chain with $F$ as its invariant distribution, we specify $g$, and are interested in estimating $\E_Fg$.  We consider the finite sample performance (based on 1000 independent replications) of the relative standard deviation fixed-volume sequential stopping rules and compare them to the relative standard deviation fixed-width sequential stopping rules (see \cite{fleg:gong:2015} and the  supplementary material).  In each case we make 90\% confidence regions for various choices of $\epsilon$ and specify our choice of $n^*$ and $b_n$. The sequential stopping rules are checked at 10\% increments of the current Monte Carlo sample size.

\subsection{Bayesian Logistic Regression}
\label{sec:bayes_log}

We continue the Bayesian logistic regression example of 
Section \ref{sub:illust_example}. Recall that a random walk
Metropolis-Hastings algorithm was implemented to sample from the
intractable posterior.  We prove the chain is geometrically ergodic in
the supplementary material.

\begin{theorem}
\label{thm:geom_erg_logistic}
The random walk based Metropolis-Hastings algorithm with invariant distribution given by the posterior from \eqref{eq:logistic model} is geometrically ergodic.
\end{theorem}

As a consequence of Theorem \ref{thm:geom_erg_logistic} and that $F$ has a moment generating function, the conditions of Theorems~\ref{thm:asymp_valid} and \ref{thm:mbm} hold.  

Motivated by the ACF plot in Figure \ref{fig:blog_trace_acf}, $b_n$ was set to $\lfloor n^{1/2} \rfloor$ and $n^* = 1000$. For calculating coverage probabilities, we declare the ``truth'' as the posterior mean from an independent simulation of length $10^9$. The results are presented in Table \ref{tab:blog_all}. As before, the univariate uncorrected method has poor coverage probabilities. For $\epsilon = 0.02$ and $0.01$, the coverage probabilities for both the mBM and uBM-Bonferroni regions are at 90\%. However, termination for mBM is significantly earlier.
\begin{table}[h]
\footnotesize 
    \caption{\footnotesize \label{tab:blog_all}Bayesian Logistic Regression: Over 1000 replications, we present termination iterations, effective sample size at termination and coverage probabilities at termination for each corresponding method. Standard errors are in parenthesis.}

\begin{center}
  \begin{tabular}{c|ccc}
\hline
 & mBM & uBM-Bonferroni &  uBM\\ \hline

          \multicolumn{4}{c}{Termination Iteration} \\ 
          \hline
$\epsilon= 0.05$  & 133005  \tiny{(196)}& 201497 \tiny{(391)} &  100445    \tiny{(213)}\\
$\epsilon = 0.02$ & 844082 \tiny{(1158)} & 1262194 \tiny{(1880)} & 629898   \tiny{(1036)}\\
$\epsilon = 0.01$ & 3309526  \tiny{(1837)} &  5046449    \tiny{(7626)} & 2510673   \tiny{(3150)}\\
 \hline
  \multicolumn{4}{c}{Effective Sample Size} \\  \hline

$\epsilon = 0.05$ &  7712    \tiny{(9)}   &   9270 \tiny{(13)} &   4643 \tiny{(7)}\\
$\epsilon = 0.02$ & 47862    \tiny{(51)}   &  57341 \tiny{(65)} & 28768 \tiny{(36)}\\
$\epsilon = 0.01$ & 186103   \tiny{(110)}  &  228448 \tiny{(271)}&113831 \tiny{(116)}\\ \hline



  \multicolumn{4}{c}{Coverage Probabilities} \\ \hline

$\epsilon = 0.05$ &  0.889  \tiny{(0.0099)} &  0.909 \tiny{(0.0091)}  & 0.569  \tiny{(0.0157)}\\
$\epsilon = 0.02$ &  0.896  \tiny{(0.0097)} &  0.912 \tiny{(0.0090)} & 0.606  \tiny{(0.0155)}\\
$\epsilon = 0.01$ &  0.892  \tiny{(0.0098)} &  0.895 \tiny{(0.0097)}  &0.606  \tiny{(0.0155)}\\ \hline
  \end{tabular}
\end{center}
\end{table}

Recall from Theorem \ref{thm:asymp_valid}, as $\epsilon$ decreases to zero, the coverage probability of confidence regions created at termination using the relative standard deviation fixed-volume sequential stopping rule converges to the nominal level. This is demonstrated in Figure \ref{fig:blog_asymp_validity} where we present the coverage probability over 1000 replications as $-\epsilon$ increases (or $\epsilon$ decreases). Notice that the increase in coverage probabilities need not be monotonic due to the underlying randomness.
\begin{figure}[h]
  \centering
\includegraphics[width = 4in]{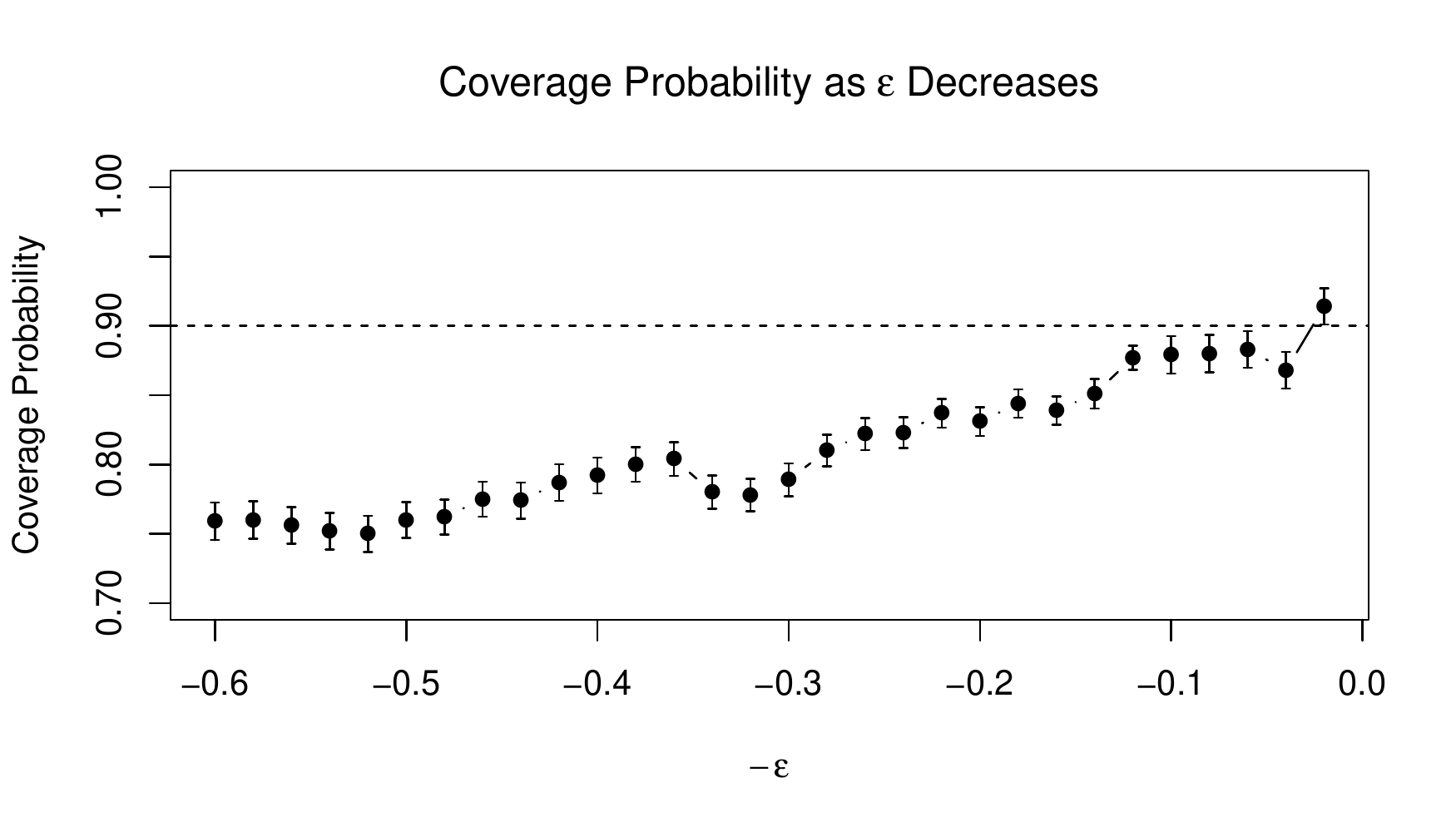}
  \caption{\footnotesize Bayesian Logistic: Plot of coverage probability with confidence bands as $\epsilon$ decreases at 90\% nominal rate. Replications = 1000.}
  \label{fig:blog_asymp_validity}
\end{figure}


\subsection{Vector Autoregressive Process}
\label{sec:var}

Consider the vector autoregressive process of order 1 (VAR(1)). For $ t = 1, 2, \dots $,
\[ Y_t = \Phi Y_{t-1} + \epsilon_t,  \]
where $Y_t \in \mathbb{R}^p$, $\Phi$ is a $p \times p$ matrix, $\epsilon_t \overset{iid}{\sim} N_p(0, \Omega)$, and $\Omega$ is a $p \times p$ positive definite matrix. The matrix $\Phi$ determines the nature of the autocorrelation.  This Markov chain has invariant  distribution $F = N_p(0, V)$ where  $vec(V) = (I_{p^2} - \Phi \otimes \Phi)^{-1} vec(\Omega)$, $\otimes$ denotes the Kronecker product, and is geometrically ergodic when the spectral radius of $\Phi$ is less than 1 \citep{tjos:1990} .

Consider the goal of estimating the mean of $F$, $\E_FY = 0$ with $\bar{Y}_n$. Then 
\[ \Sigma = (I_p - \Phi)^{-1}V + V (I_p - \Phi)^{-1} -V. \] Let $p = 5$, $\Phi = \text{diag}(.9, .5, .1, .1, .1)$, and $\Omega$ be the AR(1) covariance matrix with autocorrelation $0.9$. Since the first eigenvalue of $\Phi$ is large, the first component mixes slowest. We sample the process for $10^5$ iterations and in Figure \ref{fig:var_acf_trace} present the ACF plot for $Y^{(1)}$ and $Y^{(3)}$ and the cross-correlation (CCF) plot between $Y^{(1)}$ and $Y^{(3)}$ in addition to the trace plot for $Y^{(1)}$. Notice that $Y^{(1)}$ exhibits higher autocorrelation than $Y^{(3)}$ and there is significant cross-correlation between $Y^{(1)}$ and $Y^{(3)}$.

\begin{figure}[h]
\centering 
\subfloat[]{\includegraphics[width = 2.5in]{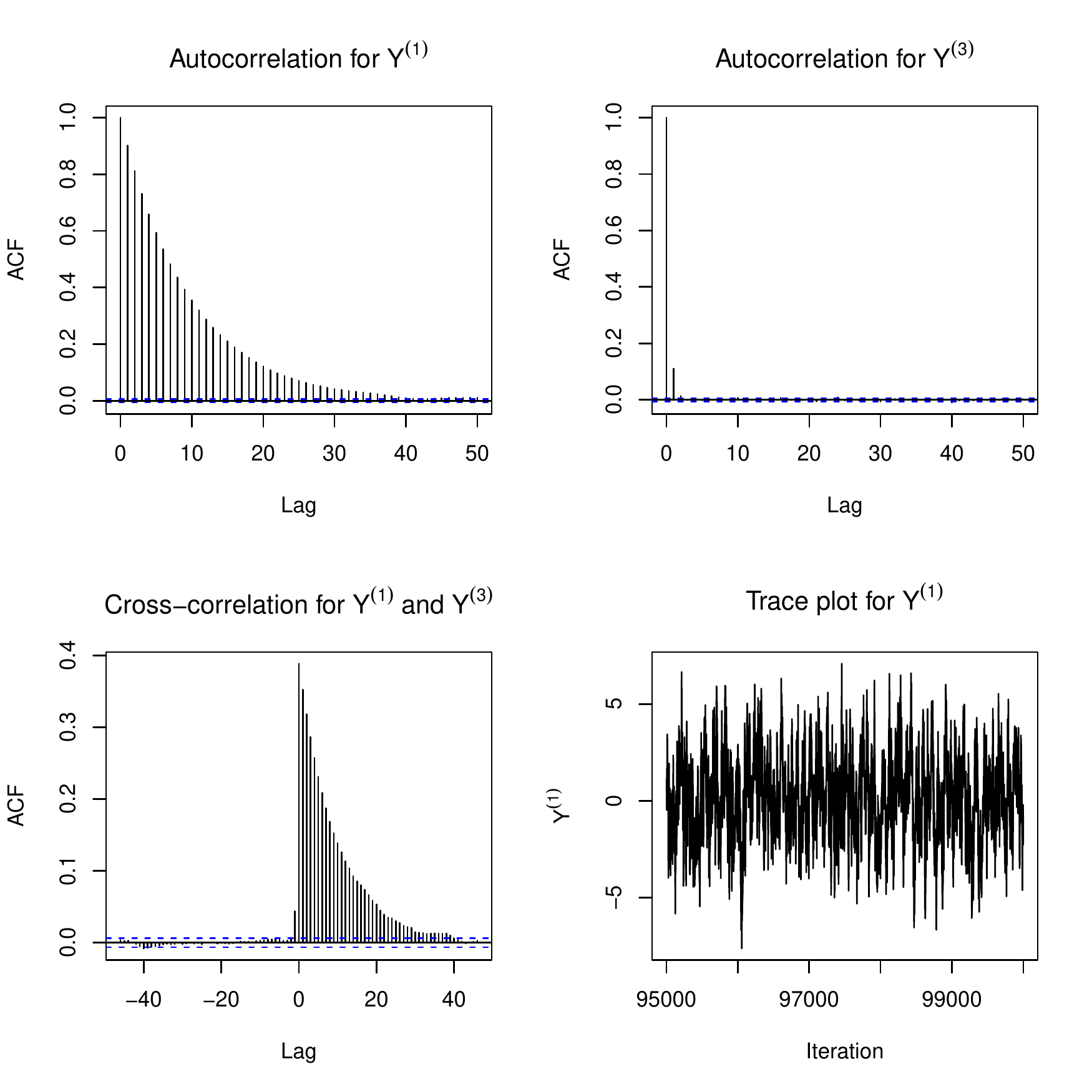}   \label{fig:var_acf_trace}}
\subfloat[]{\includegraphics[width = 2.5in]{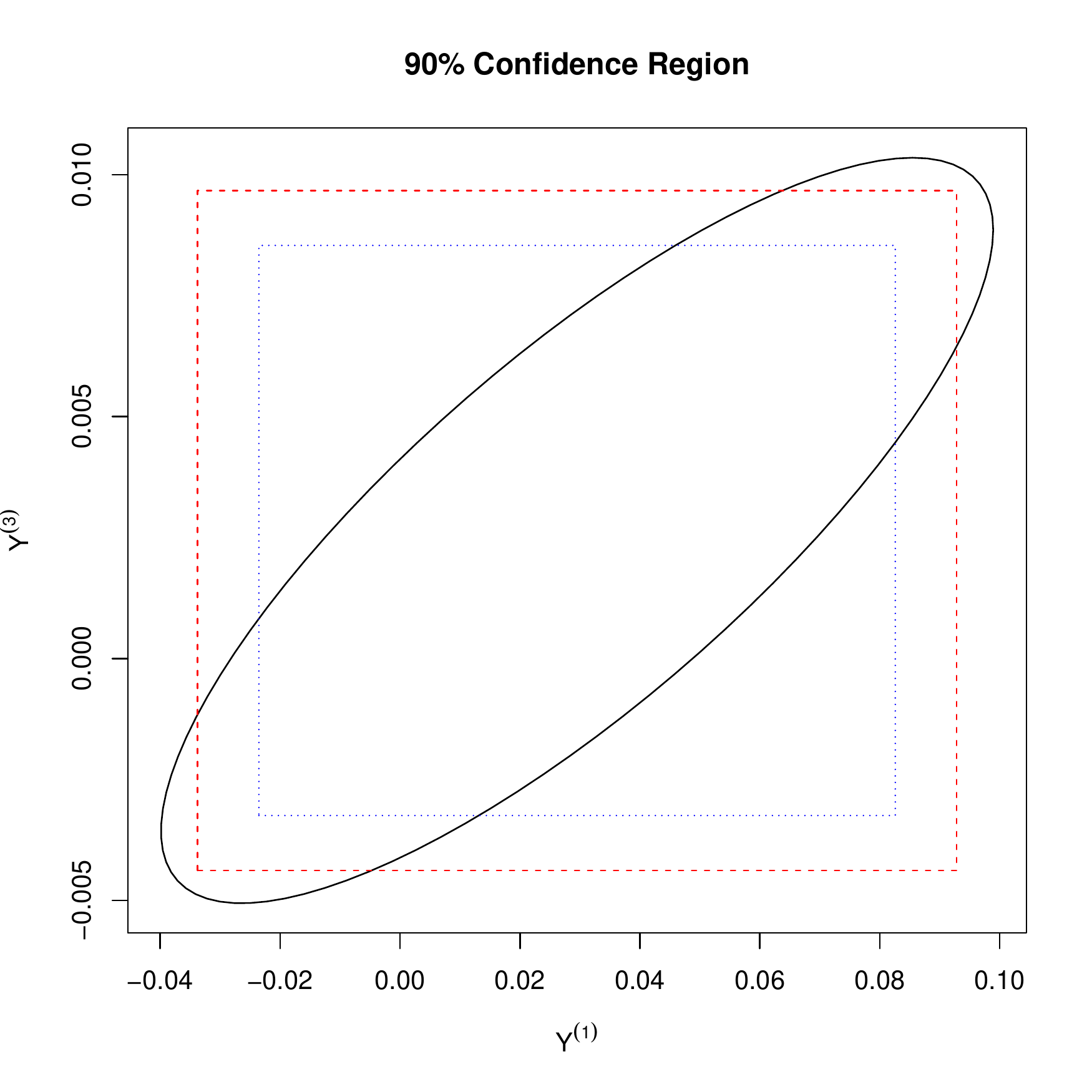}  \label{fig:ellipse_conf_region_var}}

\caption{\footnotesize VAR: (a) ACF plot for $Y^{(1)}$ and $Y^{(3)}$, CCF plot between $Y^{(1)}$ and $Y^{(3)}$, and trace plot for $Y^{(1)}$. Monte Carlo sample size is $10^5$. (b) Joint 90\% confidence region for first the two components of $Y$. The solid ellipse is made using mBM, the dotted box using uBM uncorrected and the dashed line using uBM corrected by Bonferroni. The Monte Carlo sample size is $10^5$. }
\end{figure}


Figure \ref{fig:ellipse_conf_region_var} displays joint  confidence regions for $Y^{(1)}$ and $Y^{(3)}$. Recall that the true mean is $(0,0)$, and is present in all three regions, but the ellipse produced by mBM has significantly smaller volume than the uBM boxes. The orientation of the ellipse is determined by the cross-correlations witnessed in Figure \ref{fig:var_acf_trace}.


We set $n^* = 1000$, $b_n = \lfloor n^{1/2} \rfloor$, $\epsilon$ in $\{0.05, 0.02, 0.01\}$ and at termination of each method, calculate the coverage probabilities and effective sample size. Results are presented in Table \ref{tab:var_all}. Note that as $\epsilon$ decreases, termination time increases and coverage probabilities tend to the 90\% nominal for each method. 
Also note that the uncorrected methods produce confidence regions with undesirable coverage probabilities and thus are not of interest. Consider $\epsilon = .02$ in Table \ref{tab:var_all}. Termination for mBM is at 8.8e4 iterations compared to 9.6e5 for uBM-Bonferroni. However, the estimates for multivariate ESS at 8.8e4 iterations is 4.7e4 samples compared to univariate ESS of 5.6e4 samples for 9.6e5 iterations. This is because the leading component $Y^{(1)}$ mixes much slower than the other components and defines the behavior of the univariate ESS.

\begin{table}[h]
\footnotesize 
    \caption{\footnotesize \label{tab:var_all}VAR: Over 1000 replications, we present termination iterations, effective sample size at termination and coverage probabilities at termination for each corresponding method. Standard errors are in parentheses.}
\begin{center}
  \begin{tabular}{c|ccc}
\hline
 & mBM & uBM-Bonferroni &  uBM\\ \hline

          \multicolumn{4}{c}{Termination Iteration} \\ 
          \hline
$\epsilon= 0.05$  & 14574 \tiny{(27)} & 169890 \tiny{(393)}  & 83910 \tiny{(222)}\\
$\epsilon = 0.02$ & 87682 \tiny{(118)} & 1071449 \tiny{(1733)} &   533377  \tiny{(1015)}\\
$\epsilon = 0.01$ & 343775 \tiny{(469)} &  4317599      \tiny{(5358)} & 2149042  \tiny{(3412)}\\
 \hline
  \multicolumn{4}{c}{Effective Sample Size} \\  \hline

$\epsilon = 0.05$ &  8170    \tiny{(11)}   &    9298        \tiny{(13)}  & 4658     \tiny{(7)}\\
$\epsilon = 0.02$ & 48659    \tiny{(50)}   &  57392        \tiny{(68)} &  28756    \tiny{(37)}\\
$\epsilon = 0.01$ & 190198   \tiny{(208)}  &   228772       \tiny{(223)} & 114553   \tiny{(137)}\\ \hline



\multicolumn{4}{c}{Coverage Probabilities} \\ \hline

$\epsilon = 0.05$ &  0.911  \tiny{(0.0090)}  &     0.940      \tiny{(0.0075)}  &  0.770  \tiny{(0.0133)}\\
$\epsilon = 0.02$ &  0.894  \tiny{(0.0097)}  &     0.950      \tiny{(0.0069)} &   0.769  \tiny{(0.0133)}\\
$\epsilon = 0.01$ &  0.909  \tiny{(0.0091)}   &    0.945      \tiny{(0.0072)}  &  0.779  \tiny{(0.0131)}\\ \hline
  \end{tabular}
  \end{center}
\end{table}

A small study presented in Table \ref{tab:var_ess} elaborates on this behavior. We present the mean estimate of ESS using multivariate and univariate methods based on 100 replications of Monte Carlo sample sizes $10^5$ and $10^6$. The estimate of ESS for the first component is significantly smaller than all other components leading to a conservative univariate estimate of ESS.

\begin{table}[h]
\footnotesize 
  \caption{\footnotesize  \label{tab:var_ess}VAR: Effective sample sample (ESS) estimated using proposed multivariate method and the univariate method of \cite{gong:fleg:2015} for Monte Carlo sample sizes of $n = 1e5, 1e6$ and 100 replications. Standard errors are in parentheses.}
\begin{center}
  \begin{tabular}{c|c|ccccc}
  \hline
  $n$ & ESS & ESS$_1$ & ESS$_2$ & ESS$_3$ & ESS$_4$ & ESS$_5$\\ \hline
1e5 & 55190 \tiny{(200)} & 5432 \tiny{(41)} & 33707 \tiny{(280)} & 82485 \tiny{(728)} & 82903 \tiny{(731)} & 82370 \tiny{(726)}\\ \hline
1e6 & 551015 \tiny{(945)} & 53404 \tiny{(193)} & 334656 \tiny{(1345)} & 819449 \tiny{(3334)} & 819382 \tiny{(3209)} & 819840 \tiny{(2858)}\\ \hline
  \end{tabular}
  \end{center}
\end{table}

\subsection{Bayesian Lasso} 
\label{sub:bayesian_lasso}

Let $y$ be a $K \times 1$ response vector and $X$ be a $K \times r$ matrix of predictors. We consider the following Bayesian lasso formulation of \cite{park:cas:2008}. 
\begin{align*}
y| \beta, \sigma^2, \tau^2 & \sim N_K(X\beta, \sigma^2 I_n) \\
\beta | \sigma^2, \tau^2 & \sim N_r(0, \sigma^2 D_{\tau}) \quad \text{where} \quad D_{\tau} = \text{diag}(\tau^2_1, \tau^2_2, \dots, \tau^2_r)\\
\sigma^2 & \sim \text{Inverse-Gamma}(\alpha, \xi)\\
\tau^2_{j} & \overset{iid}{\sim} \text{Exponential} \left( \dfrac{\lambda^2}{2} \right) \quad \text{ for } j = 1, \dots, r,
\end{align*}
where $\lambda, \alpha$, and $\xi$ are fixed and the Inverse-Gamma$(a,b)$ distribution with density proportional to $x^{-a -1} e^{-b/x}$.  We use a deterministic scan Gibbs sampler to draw approximate samples from the posterior; see \citet{khare:hobe:2013} for a full description of the algorithm. \cite{khare:hobe:2013} showed that for $K \geq 3$, this Gibbs sampler is geometrically ergodic for arbitrary $r$, $X$, and $\lambda$. 

We fit this model to the cookie dough dataset of \cite{osb:fearn:mill:doug:1984}. The data was collected to test the feasibility of near infra-red (NIR) spectroscopy for measuring the composition of biscuit dough pieces. There are 72 observations. The response variable is the amount of dry flour content measured and the predictor variables are 25 measurements of spectral data spaced equally between 1100 to 2498 nanometers. Since we are interested in estimating the posterior mean for $(\beta, \tau^2, \sigma^2)$, $p = 51$. The data is available in the \texttt{R} package \texttt{ppls}, and the Gibbs sampler is implemented in function \texttt{blasso} in \texttt{R} package \texttt{monomvn}. The ``truth'' was declared by averaging posterior means from 1000 independent runs each of length 1e6. We set $n^* = $ 2e4 and $b_n = \lfloor n^{1/3} \rfloor$.

In Table \ref{tab:blasso_all} we present termination results from 1000 replications. With $p$ = 51, the uncorrected univariate regions produce confidence regions with low coverage probabilities. The uBM-Bonferroni and mBM provide competitive coverage probabilities at termination. However, termination for mBM is significantly earlier than univariate methods over all values of $\epsilon$. For $\epsilon = .05$ and $.02$ we observe zero standard error for termination using mBM since termination is achieved at the same 10\% increment over all 1000 replications. Thus the variability in those estimates is less than $10\%$ of the size of the estimate.

\begin{table}[h]
\footnotesize 
  \caption{\footnotesize \label{tab:blasso_all}Bayesian Lasso: Over 1000 replications, we present termination iterations, effective sample size at termination and coverage probabilities at termination for each corresponding method. Standard errors are in parentheses.}

\begin{center}
  \begin{tabular}{c|ccc}
\hline
 & mBM & uBM-Bonferroni &  uBM\\ \hline

          \multicolumn{4}{c}{Termination Iteration} \\ 
          \hline
$\epsilon= 0.05$  & 20000  \tiny{(0)}& 69264 \tiny{(76)} &  20026    \tiny{(7)}\\
$\epsilon = 0.02$ & 69045 \tiny{(0)} & 445754 \tiny{(664)} & 122932   \tiny{(103)}\\
$\epsilon = 0.01$ & 271088  \tiny{(393)} &  1765008    \tiny{(431)} &508445   \tiny{(332)}\\
 \hline
  \multicolumn{4}{c}{Effective Sample Size} \\  \hline

$\epsilon = 0.05$ &  15631    \tiny{(4)}   &   16143  \tiny{(15)}&   4778     \tiny{(6)}\\
$\epsilon = 0.02$ & 52739    \tiny{(8)}   &  101205   \tiny{(122)} & 28358    \tiny{(24)}\\
$\epsilon = 0.01$ & 204801   \tiny{(283)}  &  395480  \tiny{(163)}&115108    \tiny{(74)}\\ \hline



  \multicolumn{4}{c}{Coverage Probabilities} \\ \hline

$\epsilon = 0.05$ &  0.898  \tiny{(0.0096)}   &    0.896      \tiny{(0.0097)}  & 0.010  \tiny{(0.0031)}\\
$\epsilon = 0.02$ &  0.892  \tiny{(0.0098)}    &   0.905      \tiny{(0.0093)} &  0.009  \tiny{(0.0030)}\\
$\epsilon = 0.01$ &  0.898  \tiny{(0.0096)}   &    0.929     \tiny{(0.0081)}  & 0.009  \tiny{(0.0030)}\\ \hline
  \end{tabular}
  \end{center}
\end{table}

\subsection{Bayesian Dynamic Spatial-Temporal Model}
\label{sec:spatio_example}

\cite{gelf:ban:gam:2005} propose a Bayesian hierarchical model for modeling univariate and multivariate dynamic spatial data viewing time as discrete and space as continuous. The methods in their paper have been implemented in the R package \texttt{spBayes}. We present a simpler version of the dynamic model as described by \cite{fin:ban:gel:2015}.

Let $s = 1, 2,  \dots, N_s$ be location sites, $t = 1, 2,  \dots, N_t$ be time-points, and the observed measurement at location $s$ and time $t$ be denoted by $y_t(s)$. In addition, let $x_t(s)$ be the $r \times 1$ vector of predictors, observed at location $s$ and time $t$, and $\beta_t$ be the $r \times 1$ vector of coefficients. For $t = 1, 2, \dots, N_t$,
\begin{align*}
y_t(s) & = x_t(s)^T \beta_t + u_t(s) + \epsilon_t(s), \quad \epsilon_t(s) \overset{ind}{\sim} N(0, \tau^2_t); \numberthis \label{eq:spatio_measurement}\\
\beta_t & = \beta_{t-1} + \eta_t, \quad \eta_t \overset{iid}{\sim} N(0, \Sigma_{\eta});\\ \numberthis \label{eq:spatio_transition}
u_t(s) & = u_{t-1}(s) + w_t(s), \quad w_t(s) \overset{ind}{\sim} GP(0, \sigma^2_t \rho(\cdot; \phi_t)),
\end{align*}
where $GP(0, \sigma^2_t \rho(\cdot; \phi_t))$ denotes a spatial Gaussian process with covariance function $\sigma^2_t \rho(\cdot; \phi_t)$. Here, $\sigma_t^2$ denotes the spatial variance component and $\rho(\cdot, \phi_t)$ is the correlation function with exponential decay.
Equation \eqref{eq:spatio_measurement} is referred to as the measurement equation and $\epsilon_t(s)$ denotes the measurement error, assumed to be independent of location and time. Equation \eqref{eq:spatio_transition}  contains the transition equations which emulate the Markovian nature of dependence in time. To complete the Bayesian hierarchy, the following priors are assumed
\begin{align*}
\beta_0 \sim N(m_0, C_0) & \quad \text{ and } \quad u_0(s) \equiv 0;\\
\tau^2_t \sim \text{IG}(a_{\tau}, b_{\tau}) & \quad \text{ and } \quad \sigma^2_t \sim \text{IG}(a_s, b_s);\\
\Sigma_{\eta} \sim \text{IW}(a_{\eta}, B_{\eta}) & \quad \text{ and } \quad \phi_t \sim \text{Unif}(a_{\phi}, b_{\phi}) \, ,
\end{align*}
where IW is the Inverse-Wishart distribution with probability density function proportional to $|\Sigma_{\eta}| ^{-\frac{a_{\eta} + q + 1}{2}} e^{-\frac{1}{2} tr(B_{\eta} \Sigma_{\eta}^{-1} )} $ and IG$(a, b)$ is the Inverse-Gamma distribution with density proportional to $x^{-a -1} e^{-b/x}$. We fit the model to the \texttt{NETemp} dataset in the \texttt{spBayes} package. This dataset contains monthly temperature measurements from 356 weather stations on the east coast of the USA collected from January 2000 to December 2010. The elevation of the weather stations is also available as a covariate. We choose a subset of the data with 10 weather stations for the year 2000, and fit the model with an intercept. The resulting posterior has $p$ = 185 components.

A component-wise Metropolis-Hastings sampler \citep{john:jone:neat:2013, jone:robe:rose:2014} is described in \cite{gelf:ban:gam:2005} and implemented in the \texttt{spDynLM} function. Default hyper parameter settings were used. The posterior and the rate of convergence for this sampler have not been studied; thus we do not know if the conditions of our theoretical results are satisfied. Our goal is to estimate the posterior expectation of $\theta = (\beta_t, u_t(s), \sigma_t^2, \Sigma_{\eta}, \tau_t^2, \phi_t)$. The truth was declared by averaging over 1000 independent runs of length 2e6 MCMC samples. We set $b_n = \lfloor n^{1/2} \rfloor$ and $n^*$ = 5e4 so that $a_n > p$ to ensure positive definitiveness of $\Sigma_n$.

Due to the Markovian transition equations in \eqref{eq:spatio_transition}, the $\beta_t$ and $u_t$ exhibit a significant covariance in the posterior distribution. This is evidenced in Figure \ref{fig:spatio_conf_region} where for Monte Carlo sample size $n = 10^5$, we present confidence regions for $\beta_1^{(0)}$ and $\beta_2^{(0)}$ which are the intercept coefficient for the first and second months, and for $u_1(1)$ and $u_2(1)$ which are the additive spatial coefficients for the first two weather stations. The thin ellipses indicate that the principle direction of variation is due to the correlation between the components. This significant reduction in volume, along with the conservative Bonferroni correction ($p$ = 185) results in increased delay in termination when using univariate methods. For smaller values  of $\epsilon$ it was not possible to store the MCMC output in memory on a 8 gigabyte machine using uBM-Bonferroni methods.
\begin{figure}[h]
  \centering
\includegraphics[width = 3.8in]{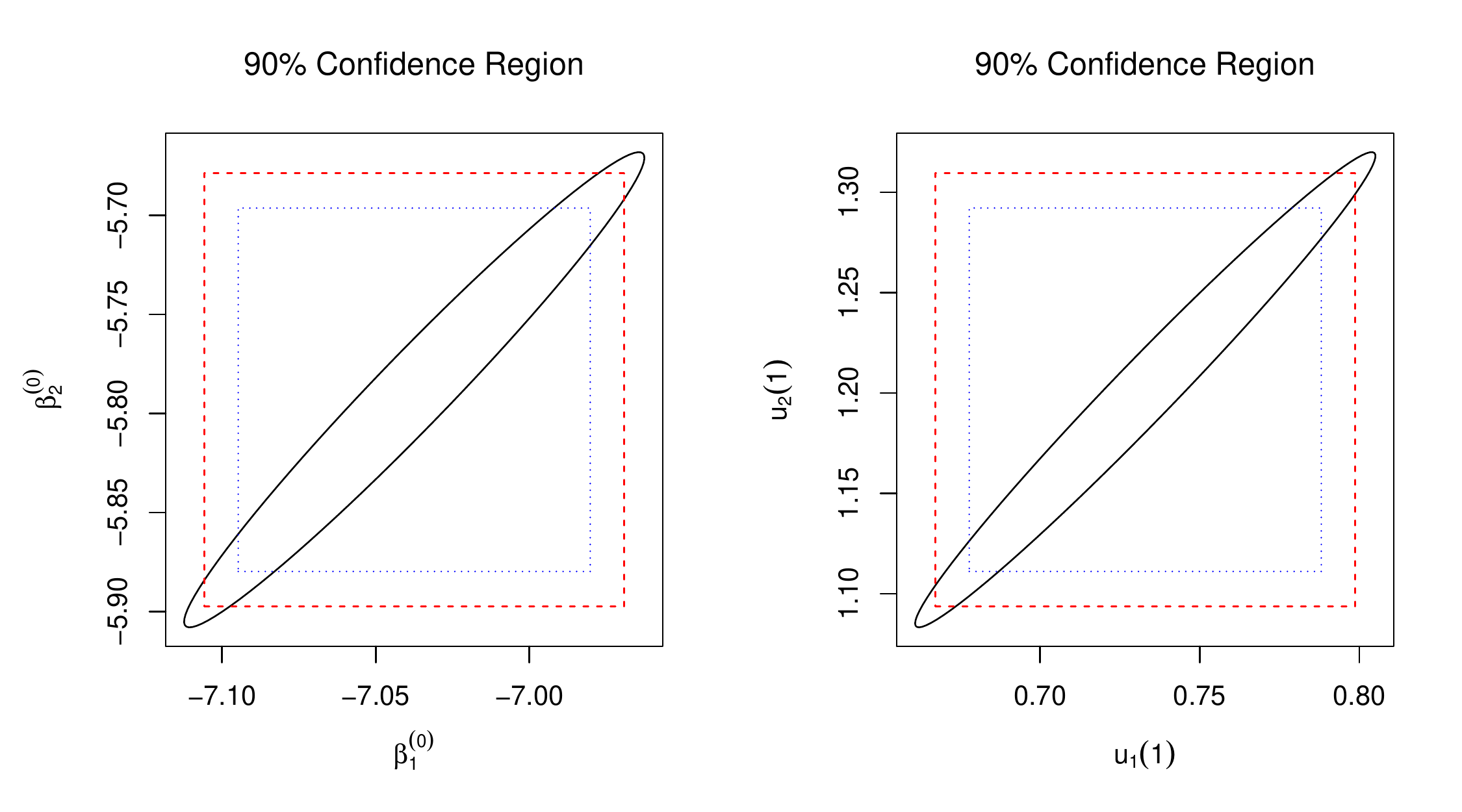}
  \caption{\footnotesize Bayesian Spatial: 90\% confidence regions for $\beta_1^{(0)}$ and $\beta_2^{(0)}$ and $u_1(1)$ and $u_2(1)$. The Monte Carlo sample size = $10^5$.}
  \label{fig:spatio_conf_region}
\end{figure}
As a result (see Table \ref{tab:spatio_all}), the univariate methods could not be implemented for smaller $\epsilon$ values. For $\epsilon = .10$, termination for mBM was at $n^* = $5e4 for every replication. At these minimum iterations, the coverage probability for mBM is at 88\%, whereas both the univariate methods have far lower coverage probabilities at 0.62 for uBM-Bonferroni and 0.003 for uBM. The coverage probabilities for the uncorrected methods are quite small since we are making 185 confidence regions simultaneously.

\begin{table}[h]
\footnotesize 
    \caption{\footnotesize \label{tab:spatio_all}Bayesian Spatial: Over 1000 replications, we present termination iteration, effective sample size at termination and coverage probabilities at termination for each corresponding method at 90\% nominal levels. Standard errors are in parentheses.}
\begin{center}
  \begin{tabular}{c|ccc}
\hline
 & mBM & uBM-Bonferroni &  uBM\\ \hline

          \multicolumn{4}{c}{Termination Iteration} \\ 
          \hline
$\epsilon= 0.10$  & 50000  \tiny{(0)}& 1200849 \tiny{(28315)} &  311856    \tiny{(491)}\\
$\epsilon= 0.05$  & 50030  \tiny{(12)}& - &  1716689 \tiny{(2178)}\\
$\epsilon = 0.02$ & 132748 \tiny{(174)} &-  & - \\
 \hline
  \multicolumn{4}{c}{Effective Sample Size} \\  \hline
$\epsilon = 0.10$ &  55170    \tiny{(20)}   &  3184 \tiny{(75)}  &   1130 \tiny{(1)}\\
$\epsilon = 0.05$ &  55190    \tiny{(20)}   & - & 4525 \tiny{(4)} \\
$\epsilon = 0.02$ & 105166    \tiny{(97)}   & - & - \\ \hline


  \multicolumn{4}{c}{Coverage Probabilities} \\ \hline
$\epsilon = 0.10$ &  0.882  \tiny{(0.0102)}   &    0.625      \tiny{(0.0153)}  & 0.007  \tiny{(0.0026)}\\
$\epsilon = 0.05$ &  0.881  \tiny{(0.0102)}   &    - & 0.016 \tiny{(0.0040)} \\
$\epsilon = 0.02$ &  0.864  \tiny{(0.0108)}    &   - &  -\\ \hline
  \end{tabular}
  \end{center}
\end{table}

\section{Discussion} 
\label{sec:discussion}

Multivariate analysis of MCMC output data has received little attention. \cite{seil:1982} and \cite{chen:seila:1987} built a framework for multivariate analysis for a Markov chain using regenerative simulation. Since establishing regenerative properties for a Markov chain requires a separate analysis for every problem and will not work well in component-wise Metropolis-Hastings samplers, the application of their work is limited. \cite{paul:mac:2012} used a specific multivariate Markov chain CLT for their MCMC convergence diagnostic.  More recently, \cite{vats:fleg:jones:2017} showed strong consistency of the multivariate spectral variance (mSV) estimators of $\Sigma$, which could potentially substitute for the mBM estimator in applications to termination rules. However, outside of toy problems where $p$ is small, the mSV estimator is computationally demanding compared to the mBM estimator. To compare these two estimators, we extend the VAR(1) example discussed in Section \ref{sec:var}, to $p = 50$. Since in this case we know the true $\Sigma$, we assess the relative error in estimation $\|\hat{\Sigma} - \Sigma\|_F/\|\Sigma\|_F$ where $\hat{\Sigma}$ is either the mBM or mSV estimator and $\|\cdot\|_F$ is the Frobenius norm. The results are reported in Table \ref{tab:comparison_msve}. The mSV estimator overall has slightly smaller relative error, but as the Monte Carlo sample size increases, computation time is significantly higher than the mBM estimator. Also note that the relative error for both estimators decreases with an increase in Monte Carlo sample size. The mSV and mBM estimators along with multivariate ESS have been implemented in the \texttt{mcmcse} R package \citep{fleg:hugh:vats:2015}. 
 \begin{table}[h]
 \footnotesize 
    \caption{\footnotesize  \label{tab:comparison_msve}  Relative error and computation time (in seconds) comparison between mSV and mBM estimators for a VAR(1) model with $p = 50$ over varying Monte Carlo samples sizes. Standard errors are in parentheses.}
\begin{center}
  \begin{tabular}{lccc}
  \hline
    Method & $n$ & Relative Error & Computation Time\\ \hline
    mBM & 1e3 & 0.373 \tiny{(0.00374)}  & 0.00069 \tiny{(1.5e-05)} \\
    mSV & 1e3 & 0.371 \tiny{(0.00375)} & 0.06035 \tiny{(2.1e-05)}\\
  \hline
    mBM & 1e4 & 0.177 \tiny{(0.00205)} & 0.00376 \tiny{(1.7e-05)} \\
    mSV & 1e4 & 0.163 \tiny{(0.00197)}& 2.13193 \tiny{(1.8e-04)} \\
  \hline
    mBM & 1e5 & 0.095 \tiny{(0.00113)} & 0.04038 \tiny{(8.6e-05)} \\
    mSV & 1e5 & 0.081 \tiny{(0.00100)}& 68.2796 \tiny{(0.11416)} \\
  \hline
  \end{tabular}
  \end{center}
\end{table}


There are two aspects of the proposed methodology that will benefit from further research. First, the rate of convergence of the Markov chain affects the choice of $b_n$ through the $\lambda$ in the SIP. Aside from \cite{damerdji:1995} and \cite{fleg:jone:2010}, little work has been done in optimal batch size selection for batch means estimators. This area warrants further research both in asymptotic and finite sample optimal batch selection.  In the supplement we study the effect of different batch sizes on simulation termination using the relative standard deviation fixed-volume sequential stopping rule. We notice that a decrease in the tolerance level $\epsilon$, decreases the sensitivity of termination to the choice of $b_n$. We have found that using a large batch size such as $b_n = \lfloor n^{1/2} \rfloor$ often works well.

Second, when using the mBM estimator, a truly large $p$ requires a large Monte Carlo sample size to ensure a positive definite estimate of $\Sigma$. The mSV estimator might yield positive definite estimates at smaller sample sizes, but is far too computationally expensive (see Table \ref{tab:comparison_msve}). It could be interesting to investigate the use of dimension reduction techniques or high-dimensional asymptotics to address this problem.

\bibliographystyle{apalike}
\bibliography{mcref}

\bigskip
\begin{center}
{\large\bf SUPPLEMENTARY MATERIAL}
\end{center}

\appendix
\section{Appendix} 
\label{sec:appendix}

\section[S.]{Proof of Theorem \prettyref{Sthm:asymp_valid} }
\label{subsec:thm_asymp_valid}

By \cite{vats:fleg:jones:2017}, since $\E_F \|g\|^{2 + \delta} < \infty$ and  $\{X_t\}$ is polynomially ergodic of order $k > (1+\epsilon_1)(1 + 2/\delta)$, Condition \prettyref{Scond:msip} holds with $\gamma(n) = n^{1/2 - \lambda}$ for $\lambda > 0$.  This implies the existence of an FCLT. Now, recall that
\[ T^*(\epsilon) = \inf \left \{ n \geq 0: \text{Vol}(C_\alpha(n))^{1/p} + s(n) \leq \epsilon K_n(Y,p) \right \}.  \]
For cleaner  notation we will use $K_n$ for $K_n(Y,p)$ and $K$ for $K(Y,p)$.  First, we show that as $\epsilon \to 0$, $T^*(\epsilon) \to \infty$. Recall $s(n) = \epsilon K_nI(n < n^*) + n^{-1}$. Consider the rule,
\[t(\epsilon) = \inf\{n \geq 0: s(n) < \epsilon K_n \} = \inf \{n \geq 0: \epsilon I(n < n^*) + (K_n n)^{-1} < \epsilon \}. \]
As $\epsilon \to 0$, $t(\epsilon) \to \infty$. Since $T^*(\epsilon) > t(\epsilon)$, $T^*(\epsilon) \to \infty$ as $ \epsilon \to 0$.

Define $V(n) = \text{Vol}(C_\alpha(n))^{1/p} + s(n)$. Then
\begin{align*}
T^*(\epsilon) & = \inf \{n \geq 0: V(n) \leq  \epsilon K_n \}.
\end{align*}
Let
\[d_{\alpha,p} = \dfrac{2\pi^{p/2} }{p \Gamma(p/2)}  \left(\chi^2_{1-\alpha, p} \right)^{p/2} .\]
Since $s(n) = o(n^{-1/2})$ and $\Sigma$ is positive definite,
\begin{align*}
n^{1/2} V(n) & = n^{1/2} \left[n^{-1/2}  \left(\dfrac{2\pi^{p/2} }{p \Gamma(p/2)} \left( \dfrac{p (a_n - 1)}{(a_n - p)} F_{1-\alpha, p, a_n-p}  \right)^{p/2} |\Sigma_n |^{1/2} \right)^{1/p} + s(n) \right]\\
& = \left(\dfrac{2\pi^{p/2} }{p \Gamma(p/2)} \left( \dfrac{p (a_n - 1)}{(a_n - p)} F_{1-\alpha, p, a_n-p}  \right)^{p/2} |\Sigma_n |^{1/2} \right)^{1/p} + n^{1/2} s(n)\\
& \to (d_{\alpha,p} |\Sigma|^{1/2})^{1/p} > 0\text{ w.p. 1 as } n \to \infty. \numberthis \label{eq:vn_limit}
\end{align*}
By definition of $T^*(\epsilon)$
\begin{equation}
\label{eq:vn_upper}
  V(T^*(\epsilon) - 1) > \epsilon K_{T^*(\epsilon) - 1},
\end{equation}
and there exists a random variable $Z(\epsilon)$ on $[0,1]$ such that,
\begin{equation}
\label{eq:vn_lower}
V(T^*(\epsilon) + Z(\epsilon) ) \leq \epsilon K_{T^*(\epsilon) + Z(\epsilon)}.
\end{equation}
Since $K_n$ is strongly consistent for $K$ and $T^*(\epsilon) \to \infty$ w.p. 1 as $\epsilon \to 0$,
\begin{equation}
\label{eq:lambda_const}
K_{T^*(\epsilon)}\to K \quad \text{w.p. 1},
\end{equation}
Using \eqref{eq:vn_limit}, \eqref{eq:vn_upper}, \eqref{eq:vn_lower}, and \eqref{eq:lambda_const}
\begin{align*}
\limsup_{\epsilon \to 0} \epsilon T^*(\epsilon)^{1/2} \leq \limsup_{\epsilon \to 0} \dfrac{T^*(\epsilon)^{1/2} V(T^*(\epsilon) - 1)}{K_{T^*(\epsilon) - 1}}  = d_{\alpha,p}^{1/p} \dfrac{|\Sigma|^{1/2p}}{K}   \quad \text{w.p. 1}\\
\liminf_{\epsilon \to 0} \epsilon T^*(\epsilon)^{1/2} \geq \liminf_{\epsilon \to 0} \dfrac{T^*(\epsilon)^{1/2}V(T^*(\epsilon) + Z(\epsilon))}{ K_{T^*(\epsilon) + Z(\epsilon)}} = d_{\alpha,p}^{1/p} \dfrac{|\Sigma|^{1/2p}}{K} \quad \text{w.p. 1}.
\end{align*}
Thus, 
\begin{equation}
\label{eq:eps_teps}
\lim_{\epsilon \to 0} \epsilon T^*(\epsilon)^{1/2} = d_{\alpha,p}^{1/p} \dfrac{|\Sigma|^{1/2p}}{K}.
\end{equation}

 Using \eqref{eq:eps_teps} and the existence of a FCLT, by a  standard random-time-change argument \cite[p. 151]{bill:1968}
\begin{equation*}
 \label{eq:random_time_change}
 \sqrt{T^*(\epsilon)} \Sigma^{-1/2}_{T^*(\epsilon)} (\theta_{T^*(\epsilon)} - \theta) \overset{d}{\to} N_p(0,I_p) \quad \text{ as } \epsilon \to 0.
 \end{equation*}

Finally,
\begin{align*}
& \Pr[\theta \in C_\alpha(T^*(\epsilon))] \\
& = \Pr\left[ T^*(\epsilon) (\theta_{T^*(\epsilon)} - \theta)^T \Sigma^{-1}_{T(\epsilon)} (\theta_{T^*(\epsilon)} - \theta)  \leq \dfrac{p (a_{T^*(\epsilon)} - 1)}{(a_{T^*(\epsilon)} - p)} F_{1-\alpha, p, a_{T^*(\epsilon)}-p}  ; |\Sigma_{T^*(\epsilon)}| \ne 0 \right]\\
 & \quad + \Pr\left[\theta \in C_\alpha(T^*(\epsilon)); |\Sigma_{T^*(\epsilon)}| = 0 \right] \\
 & \to 1 - \alpha \text{ w.p. 1 as } n \to \infty \text{ since }  \Pr(|\Sigma_{T^*(\epsilon)}| = 0 ) \to 0 \text{ as } \epsilon \to 0. 
\end{align*}

\section{Proof of Theorem \ref{Sthm:mbm} }
\label{subsec:thm1}

Note that in this section we use the notation $| \cdot |$ to denote absolute value and not determinant. We first state the theorem more generally for processes that satisfy a strong invariance principle. Then, the proof of Theorem \ref{Sthm:mbm} will follow from Theorem \ref{thm:mbm_general} below and Corollary 5 in \cite{vats:fleg:jones:2017}.

\begin{theorem} \label{thm:mbm_general} Let Conditions \prettyref{Scond:msip} and \prettyref{Scond:bn_conditions} hold. If $b_n^{-1/2} (\log n)^{1/2} \gamma(n) \to 0$ as $n \to \infty$, then $\Sigma_{n} \to \Sigma$ w.p. 1 as $n \to \infty$.  \end{theorem}

The proof of this theorem will be presented in a series of lemmas. We first construct $\widetilde{\Sigma}$, a Brownian motion equivalent of the batch means estimator and show that $\widetilde{\Sigma}$ converges to the identity matrix with probability 1 as $n$ increases. This result will be critical in proving the theorem.

Let $B(t)$ be a $p$-dimensional standard Brownian motion, and for $i = 1, \dots, p$, let $B^{(i)}$ denote the $i$th component univariate Brownian motion. For $k = 0, \dots, a_n - 1$ define 
\[\bar{B}^{(i)}_k = \dfrac{1}{b_n} (B^{(i)}((k+1)b_n) - B^{(i)}(kb_n)) \quad \text{ and } \quad \bar{B}^{(i)}(n) = \dfrac{1}{n} B^{(i)}(n). \] 
For $\bar{B}_k = ( \bar{B}^{(1)}_k, \dots, \bar{B}^{(p)}_k )^T$ and  $\bar{B}(n) = \left( \bar{B}^{(1)}(n), \dots, \bar{B}^{(p)}(n) \right)^T$, define
\[ \widetilde{\Sigma} = \dfrac{b_n}{a_n - 1} \ds \sum_{k=0}^{a_n - 1} [\bar{B}_k - \bar{B}(n)][\bar{B}_k - \bar{B}(n)]^T. \]

Here $\widetilde{\Sigma}$ is the Brownian motion equivalent of $\Sigma_n$ and in Lemma \ref{lemma:sigma_tilde_i} we will show that $\widetilde{\Sigma}$ converges to the identity matrix with probability 1. But first we state some results we will need.

\begin{lemma}[\citet{kend:stua:1963}]
\label{lemma:kendall_chi}
If $Z \sim \chi^2_v$, then for all positive integers $r$ there exists a constant $K:= K(r)$ such that $\E[(Z-v)^{2r}] \leq Kv^r$.
\end{lemma}

\begin{lemma}[\citet{bill:1968}]
\label{lemma:billingsley}
Consider a family of random variables $\{ R_n: n \geq 1\}$. If $\E[|R_n|] \leq s_n$, where $s_n$ is a sequence such that $\ds \sum_{n} s_n < \infty$, then $R_n \to 0$ w.p. 1 as $n \to \infty$.
\end{lemma}

\begin{lemma}
\label{lemma:sigma_tilde_i}
If Condition \prettyref{Scond:bn_conditions} holds, then $\widetilde{\Sigma} \to I_p$ with probability 1 as $n \to \infty$ where $I_p$ is the $p \times p$ identity matrix.
\end{lemma}

\begin{proof}
For $i,j = 1, \dots, p$, let $\widetilde{\Sigma}_{ij}$ denote the $(i,j)$th component of $\widetilde{\Sigma}$. For $i = j$, \cite{dame:1994} showed that $\widetilde{\Sigma}_{ij} \to 1$ with probability 1 as $n \to \infty$. Thus it is left to show that for $i \ne j$, $\widetilde{\Sigma}_{ij} \to 0$ with probability 1 as $n \to \infty$. 
\begin{align}\label{eq:tildesig}
\begin{split}
\widetilde{\Sigma}_{ij} & = \dfrac{b_n}{a_n - 1} \ds \sum_{k=0}^{a_n - 1} \left[\bar{B}_k^{(i)} - \bar{B}^{(i)}(n)  \right]\left[\bar{B}_k^{(j)} - \bar{B}^{(j)}(n)  \right]\\
& = \dfrac{b_n}{a_n - 1} \ds \sum_{k=0}^{a_n - 1}  \left[ \bar{B}^{(i)}_k \bar{B}^{(j)}_k - \bar{B}^{(i)}_k \bar{B}^{(j)}(n) - \bar{B}^{(i)}(n)\bar{B}^{(j)}_k + \bar{B}^{(i)}(n)\bar{B}^{(j)}(n) \right].
\end{split}
\end{align}

We will show that each of the four terms in \eqref{eq:tildesig} converges to 0 with probability 1 as $n \to \infty$. But first we note that by the properties of Brownian motion, for all $k = 0, \dots, a_n - 1$, 
\begin{align*}
& \bar{B}_k^{(i)} \overset{iid}{\sim} N\left(0, \dfrac{1}{b_n} \right)
\text{ and }
\bar{B}_k^{(j)} \overset{iid}{\sim} N\left(0, \dfrac{1}{b_n} \right) \text{ independently}\\
\Rightarrow & \sqrt{b_n}\bar{B}_k^{(i)} \overset{iid}{\sim} N\left(0, 1 \right)
\text{ and }
\sqrt{b_n}\bar{B}_k^{(j)} \overset{iid}{\sim} N\left(0, 1 \right) \text{ independently}.\numberthis \label{eq:b_k_standardized}
\end{align*}

\begin{enumerate}
\item 
Naturally by \eqref{eq:b_k_standardized},
\begin{equation*}\label{normaltrick}
X_k := \sqrt{\dfrac{b_n}{2}}\bar{B}_k^{(i)} + \sqrt{\dfrac{b_n}{2}}\bar{B}_k^{(j)} \overset{iid}{\sim} N\left(0, 1 \right)
\, \text{ and } \,
Y_k := \sqrt{\dfrac{b_n}{2}}\bar{B}_k^{(i)} - \sqrt{\dfrac{b_n}{2}}\bar{B}_k^{(j)} \overset{iid}{\sim} N\left(0, 1 \right)
\end{equation*}

Notice that $AB =  (A+B)^2/4 - (A-B)^2/4$. Using $X_k$ as $(A+B)$ and $Y_k$ as $(A-B)$, we can write $ {b_n} \bar{B}^{(i)}_k \bar{B}^{(j)}_k / 2$ as a linear combination of two $\chi^2$ random variables. Specifically,
\begin{align*}
\dfrac{b_n}{a_n-1} \ds \sum_{k=0}^{a_n - 1} \bar{B}^{(i)}_k \bar{B}^{(j)}_k & = \dfrac{2}{a_n - 1} \ds \sum_{k=0}^{a_n - 1} \sqrt{\dfrac{b_n}{2}}\bar{B}^{(i)}_k  \sqrt{\dfrac{b_n}{2}}\bar{B}^{(j)}_k\\
& = \dfrac{1}{2(a_n - 1)} \ds \sum_{k=0}^{a_n - 1} [X_k^2 - Y_k^2]\\
& = \dfrac{1}{2(a_n - 1)} \ds \sum_{k=0}^{a_n - 1}X_k^2 - \dfrac{1}{2(a_n - 1)} \ds \sum_{k=0}^{a_n - 1}Y_k^2 \\
& = \dfrac{a_n}{2(a_n - 1)}  \dfrac{X}{a_n} - \dfrac{a_n}{2(a_n - 1)} \dfrac{Y}{a_n}, \numberthis \label{eq:firstterm}
\end{align*}

where $X =  \sum_{k=0}^{a_n - 1}X_k^2 \sim \chi^2_{a_n}$ and $Y =  \sum_{k=0}^{a_n - 1}Y_k^2 \sim \chi^2_{a_n}$ independently.

By Lemma \ref{lemma:kendall_chi}, for all positive integers $c$,
\begin{equation*}
\E \left[ (X - a_n)^{2c}\right] \leq K a_n^c \Rightarrow \E \left[ \left(\dfrac{X}{a_n} - 1 \right)^{2c} \right] \leq K \left(\dfrac{b_n}{n} \right)^c.
\end{equation*}

Thus by Lemma \ref{lemma:billingsley} and Condition \prettyref{Scond:bn_conditions}\prettyref{Scond:bn_by_n_c}, ${X}/{a_n} \to 1 $ with probability 1, as $n\to \infty$. Similarly, ${Y}/{a_n} \to 1 $ with probability 1, as $n\to \infty$. Using this result in \eqref{eq:firstterm} and the fact that ${a_n}/({a_n - 1}) \to 1$ as $n \to \infty$,
\[
\dfrac{b_n}{a_n-1} \ds \sum_{k=0}^{a_n - 1} \bar{B}^{(i)}_k \bar{B}^{(j)}_k \to 0 \text{ w.p. 1 as } n \to \infty.
\]
 
\item  By the definition of $\bar{B}(n)$ and $\bar{B}_k$,
\begin{align*}
\dfrac{b_n}{a_n-1} \ds \sum_{k=0}^{a_n - 1} \bar{B}_k^{(i)} \bar{B}^{(j)}(n)& = \dfrac{1}{a_n - 1} \dfrac{1}{n} B^{(j)}(n) \ds \sum_{k=0}^{a_n-1} B^{(i)}((k+1)b_n) - B^{(i)}(kb_n) \\
& = \dfrac{1}{a_n - 1} \dfrac{1}{n} B^{(j)}(n)  B^{(i)}(a_nb_n)\\
& = \dfrac{a_n}{a_n - 1} \dfrac{\sqrt{b_n}}{n} B^{(j)}(n) \dfrac{\sqrt{b_n}}{a_n b_n} B^{(i)}(a_nb_n) \; . \numberthis \label{eq:secondterm}
\end{align*}

Using properties of Brownian motion,
\begin{align*}
B^{(j)}(n) \sim N(0,n) & ~~\text{ and }~~ B^{(i)}(a_nb_n) \sim N(0, a_nb_n)\\
\Rightarrow \dfrac{\sqrt{b_n}}{n}B^{(j)}(n) \overset{d}{\sim} N\left(0,\dfrac{b_n}{n}\right)& ~~\text{ and }~~ \dfrac{\sqrt{b_n}}{a_nb_n}B^{(i)}(a_nb_n) \overset{d}{\sim} N\left(0, \dfrac{1}{a_n}\right). \numberthis \label{eq:dirac}
\end{align*}

As $n \to \infty$ both terms in \eqref{eq:dirac} tend to Dirac's delta function. Thus as $n \to \infty$.
\begin{equation}\label{eq:secondfinal}
\dfrac{\sqrt{b_n}}{n}B^{(j)}(n) \to 0 \text{~w.p. 1} \quad \text{ and }\quad \dfrac{\sqrt{b_n}}{a_nb_n}B^{(i)}(a_nb_n) \to 0 \text{ w.p. 1}.
\end{equation}

Using \eqref{eq:secondfinal} in \eqref{eq:secondterm},
\[ \dfrac{b_n}{a_n-1} \ds \sum_{k=0}^{a_n - 1} \bar{B}_k^{(i)} \bar{B}^{(j)}(n) \to 0 ~~\text{w.p. 1 as } n \to \infty. \]

\item A similar argument as above yields, 
\[ \dfrac{b_n}{a_n-1} \ds \sum_{k=0}^{a_n - 1} \bar{B}^{(i)}(n)\bar{B}_k^{(j)}  \to 0 ~~\text{w.p. 1 as } n \to \infty.  \]

\item By the definition of $\bar{B}^{(i)}(n)$,
\begin{align*}
\dfrac{b_n}{a_n-1} \ds \sum_{k=0}^{a_n-1} \bar{B}^{(i)}(n) \bar{B}^{(j)}(n)& = \dfrac{b_n}{a_n-1} a_n \dfrac{1}{n} B^{(i)}(n)\dfrac{1}{n}B^{(j)}(n)\\
& = \dfrac{a_n}{a_n-1} \dfrac{\sqrt{b_n}}{n} B^{(i)}(n) \dfrac{\sqrt{b_n}}{n} B^{(j)}(n)\\
& \to 0 ~~\text{w.p. 1 as } n \to \infty \text{ by \eqref{eq:secondfinal}}.
\end{align*}
\end{enumerate}

Thus each term in \eqref{eq:tildesig} goes to 0 with probability 1 as $n \to \infty$, yielding $\widetilde{\Sigma} \to I_p$ with probability 1 as $n \to \infty$.
\end{proof}

\begin{corollary}
\label{cor:l_sigmtilde_lt_to_sigma}
If Condition \prettyref{Scond:bn_conditions} holds, then for any conformable matrix $L$, as $n \to \infty$, $L \widetilde{\Sigma} L^T \to LL^T$ with probability 1.
\end{corollary}

For the rest of the proof, we will require some results regarding Brownian motion.

\begin{lemma}[\citet{csor:reve:1981}]
\label{bn.diff.bound}
Suppose Condition \prettyref{Scond:bn_conditions}\prettyref{Scond:bn_to_n} holds, then for all $\epsilon > 0$ and for almost all sample paths, there exists $n_0(\epsilon)$ such that for all $n \geq n_0$ and all $i = 1, \dots, p$
\[\sup_{0\leq t \leq n-b_n} \sup_{0 \leq s \leq b_n} |B^{(i)}(t+s) - B^{(i)}(t)| < (1+\epsilon) \left(2b_n \left(\log{\dfrac{n}{b_n}} + \log{\log n}  \right) \right) ^{1/2}  \]
\[\sup_{0 \leq s \leq b_n} |B^{(i)}(n) - B^{(i)}(n-s)| < (1+\epsilon) \left(2b_n \left(\log{\dfrac{n}{b_n}} + \log{\log n}  \right) \right) ^{1/2}  \] and
\[\left| B^{(i)}(n) \right|  < (1 + \epsilon) \sqrt{2n \log \log n}.\]
\end{lemma}

\begin{corollary}[\citet{dame:1994}]
Suppose Condition \prettyref{Scond:bn_conditions}\prettyref{Scond:bn_to_n} holds, then for all $\epsilon >0$ and for almost all sample paths, there exists $n_0(\epsilon)$ such that for all $n \geq n_0$ and all $i = 1, \dots, p$
\[ |\bar{B}_k^{(i)}(b_n)| \leq \dfrac{\sqrt{2}}{\sqrt{b_n}} (1+\epsilon) \left( \log \dfrac{n}{b_n} + \log \log n\right)^{1/2}. \] 
\end{corollary}

Recall $L$ in \prettyref{Seq:sip}  and set $\Sigma := LL^{T}$. Define
$C(t) = LB(t)$ and if $C^{(i)}(t)$ is the $i$th component of $C(t)$, define
\[\bar{C}^{(i)}_k  := \dfrac{1}{b_n} \left(C^{(i)}((k+1)b_n) - C^{(i)}(kb_n)
\right) \quad \text{ and } \quad \bar{C}^{(i)}(n)  := \dfrac{1}{n} C^{(i)}(n).\]
Since $C^{(i)}(t) \sim N(0, t \Sigma_{ii})$, where $\Sigma_{ii}$ is the $i$th diagonal of $\Sigma$, $C^{(i)}/\sqrt{\Sigma_{ii}}$ is a 1-dimensional Brownian motion. As a consequence, we have the following corollaries of Lemma \ref{bn.diff.bound}.

\begin{corollary}\label{cor:iter}
For all $\epsilon > 0$ and for almost all sample paths there exists $n_0(\epsilon)$ such that for all $n \geq n_0$ and all $i = 1, \dots,p$
\begin{equation*}
\left| C^{(i)} (n) \right| < (1+ \epsilon) \left[2 \Sigma_{ii} n \log \log n \right]^{1/2},
\end{equation*}
where $\Sigma_{ii}$ is the $i$th diagonal of $\Sigma$.
\end{corollary}

\begin{corollary}\label{cor:seconditer}
For all $\epsilon > 0$ and for almost all sample paths, there exists $n_0(\epsilon)$ such that for all $n \geq n_0$ and all $i = 1, \dots, p$
\begin{equation*}
\left|\bar{C}_k^{(i)}  \right| \leq \sqrt{\dfrac{2 \Sigma_{ii}}{b_n}} (1+\epsilon) \left( \log \dfrac{n}{b_n} + \log \log n \right)^{1/2},
\end{equation*}
where $\Sigma_{ii}$ is the $i$th diagonal of $\Sigma$.
\end{corollary}




We finally come to the last leg of the proof, where we will show that for the $(i,j)$th element of $\Sigma_n$, $|\Sigma_{n,ij} - \Sigma_{ij}| \to 0$ with probability 1 as $n \to \infty$. 

Recall $Y_t = g(X_t)$. Let $Y_t^{(i)}$ be the $i$th component of $Y_t$. Define for each $i = 1, \dots, p$, the process $V_l^{(i)} = Y_l^{(i)} - \theta_i$ for $l = 1, 2, \dots$. Further, for $k = 0, \dots, a_n - 1$ and $j = 1, \dots, p$ define 
\[ \bar{V}_k^{(i)} = \dfrac{1}{b_n} \ds \sum_{l=1}^{b_n} V^{(i)}_{kb_n + l} \quad  \text{ and } \quad \bar{V}^{(i)}(n) = \dfrac{1}{n} \ds \sum_{l=1}^{n} V_l^{(i)}.\]

Then
\begin{equation*}\label{mainterm}
\Sigma_{n,ij} = \dfrac{b_n}{a_n - 1} \ds \sum_{k=0}^{a_n - 1} \left[\bar{V}_k^{(i)} - \bar{V}^{(i)}(n) \right]\left[\bar{V}_k^{(j)} - \bar{V}^{(j)}(n) \right].
\end{equation*}

We will show that $\left|\Sigma_{n,ij} - \Sigma_{ij}\right| \to 0$ w.p. 1 as $n \to \infty$.
\begin{align*}
&\left|\Sigma_{n,ij} - \Sigma_{ij} \right|\\
& = \left| \dfrac{b_n}{a_n-1} \ds \sum_{k=0}^{a_n - 1} \left[\bar{V}^{(i)}_k - \bar{V}^{(i)}(n) \right]\left[\bar{V}^{(j)}_k - \bar{V}^{(j)}(n) \right] - \Sigma_{ij} \right|\\
& = \left| \dfrac{b_n}{a_n-1} \ds \sum_{k=0}^{a_n - 1} \left[\bar{V}^{(i)}_k - \bar{V}^{(i)}(n) \pm \bar{C}_k^{(i)} \pm \bar{C}^{(i)}(n) \right]\left[\bar{V}^{(j)}_k - \bar{V}^{(j)}(n)\pm \bar{C}_k^{(j)} \pm \bar{C}^{(j)}(n) \right] - \Sigma_{ij} \right|\\
& = \bigg|\dfrac{b_n}{a_n-1} \ds \sum_{k=0}^{a_n-1} \left[\left(\bar{V}_k^{(i)} - \bar{C}_k^{(i)}\right) - \left(\bar{V}^{(i)}(n) - \bar{C}^{(i)}(n)\right) + \left(\bar{C}_k^{(i)} - \bar{C}^{(i)}(n)  \right) \right]\\
& \qquad \left[\left(\bar{V}_k^{(j)} - \bar{C}_k^{(j)}\right) - \left(\bar{V}^{(j)}(n) - \bar{C}^{(j)}(n)\right) + \left(\bar{C}_k^{(j)} - \bar{C}^{(j)}(n)  \right) \right]  - \Sigma_{ij} \bigg| \\
& \leq \left|\dfrac{b_n}{a_n - 1} \ds \sum_{k=0}^{a_n-1} \left[\bar{C}_k^{(i)} - \bar{C}^{(i)}(n)  \right] \left[\bar{C}_k^{(j)} - \bar{C}^{(j)}(n)  \right] - \Sigma_{ij}\right|\\
& + \dfrac{b_n}{a_n-1} \ds \sum_{k=0}^{a_n-1} \bigg[\left|\left(\bar{V}^{(i)}_k - \bar{C}^{(i)}_k\right)\left(\bar{V}^{(j)}_k - \bar{C}^{(j)}_k\right) \right| \\
& + \left|\left(\bar{V}^{(i)}(n) - \bar{C}^{(i)}(n)\right) \left(\bar{V}^{(j)}(n) - \bar{C}^{(j)}(n)\right)\right| \\
& + \left| \left(\bar{V}^{(i)}_k - \bar{C}^{(i)}_k \right)  \left(\bar{V}^{(j)}(n) - \bar{C}^{(j)}(n) \right) \right|+ \left|\left(\bar{V}^{(i)}(n) - \bar{C}^{(i)}(n) \right)  \left( \bar{V}^{(j)}_k - \bar{C}^{(j)}_k \right)  \right|\\
& + \left| \left( \bar{V}^{(i)}_k - \bar{C}^{(i)}_k \right)\bar{C}^{(j)}_k \right| + \left| \left( \bar{V}^{(j)}_k - \bar{C}^{(j)}_k \right)\bar{C}^{(i)}_k \right|\\
& + \left| \left( \bar{V}^{(i)}_k - \bar{C}^{(i)}_k \right)\bar{C}^{(j)}(n) \right| + \left| \left( \bar{V}^{(j)}_k - \bar{C}^{(j)}_k \right)\bar{C}^{(i)}(n) \right|\\
& + \left| \left( \bar{V}^{(i)}(n) - \bar{C}^{(i)}(n) \right)\bar{C}^{(j)}_k \right| + \left| \left( \bar{V}^{(j)}(n) - \bar{C}^{(j)}(n) \right)\bar{C}^{(i)}_k \right|\\
& + \left| \left( \bar{V}^{(i)}(n) - \bar{C}^{(i)}(n) \right)\bar{C}^{(j)}(n) \right| + \left| \left( \bar{V}^{(j)}(n) - \bar{C}^{(j)}(n) \right)\bar{C}^{(i)}(n) \right| \bigg].
\end{align*}
We will show that each of the 13 terms above tends to 0 w.p. 1 as $n \to \infty$.

\begin{enumerate}[1.]
\item Notice that, \[ \dfrac{b_n}{a_n - 1} \ds \sum_{k=0}^{a_n-1} \left[\bar{C}_k^{(i)} - \bar{C}^{(i)}(n)  \right] \left[\bar{C}_k^{(j)} - \bar{C}^{(j)}(n)  \right], \]
is the $(i, j)$th entry in $L\widetilde{\Sigma}L^T$. Thus, by Corollary \ref{cor:l_sigmtilde_lt_to_sigma}, with probability 1 as $n\to \infty$,
\[\left|\dfrac{b_n}{a_n - 1} \ds \sum_{k=0}^{a_n-1} \left[\bar{C}_k^{(i)} - \bar{C}^{(i)}(n)  \right] \left[\bar{C}_k^{(j)} - \bar{C}^{(j)}(n)  \right] - \Sigma_{ij}  \right|  \to 0. \]

\item By Condition 1
\[\left\| \ds \sum_{l = 0}^{n} V_l - LB(n) \right\|< D \gamma(n)\text{ w.p. 1, } \]
where $V_l = \left(V_l^{(1)}, \dots, V_l^{(p)} \right)$. Hence, for components $i$ and $j$
\begin{equation}\label{Dbound}
\left| \ds \sum_{l=1}^{n} V_l^{(i)} - C^{(i)}(n) \right| < D \gamma(n) \quad \text{ and } \quad  \left|\ds \sum_{l=1}^{n} V_l^{(j)} - C^{(j)}(n) \right| < D \gamma(n).
\end{equation}

Note that,
\begin{align}\label{two1}
\left|\bar{V}_k^{(i)} - \bar{C}_k^{(i)}\right| &=  \left|\dfrac{1}{b_n} \left[\ds \sum_{l=1}^{(k+1)b_n} V_l^{(i)} - C^{(i)}((k+1)b_n)  \right] - \dfrac{1}{b_n} \left[\ds \sum_{l=1}^{kb_n}V_l^{(i)} - C^{(i)}(kb_n)  \right] \right| \nonumber \\
& \leq  \dfrac{1}{b_n} \left[\left| \ds \sum_{l=1}^{(k+1)b_n} V_l^{(i)} - C^{(i)}((k+1)b_n)  \right| +  \left| \ds \sum_{l=1}^{kb_n}V_l^{(i)} - C^{(i)}(kb_n)  \right| \right] \nonumber \\
& \leq \dfrac{2}{b_n} D \gamma(n).
\end{align}

Similarly 
\begin{equation}\label{two2}
\left|\bar{V}_k^{(j)} - \bar{C}_k^{(j)}\right| \leq \dfrac{2}{b_n} D \gamma(n).
\end{equation}

Thus, using (\ref{two1}) and (\ref{two2}),
\begin{align*}
\dfrac{b_n}{a_n-1} \ds\sum_{k=0}^{a_n-1} \left|\left(\bar{V}^{(i)}_k - \bar{C}^{(i)}_k\right)\left(\bar{V}^{(j)}_k - \bar{C}^{(j)}_k\right) \right|& \leq \dfrac{b_n}{a_n-1} a_n \dfrac{4D^2}{b_n^2} \gamma(n)^2\\
& \leq 4D^2 \dfrac{a_n}{a_n-1} \dfrac{\log n}{b_n} \gamma(n)^2\\
& \to 0 \text{ w.p 1 as } n \to \infty.
\end{align*}

\item  By (\ref{Dbound}), we get
\begin{align}\label{three1}
\left| \bar{V}^{(i)}(n) - \bar{C}^{(i)}(n)  \right|& = \dfrac{1}{n} \left|\ds \sum_{l=1}^{n} V_l^{(i)} - C^{(i)}(n)  \right|  < D \dfrac{\gamma(n)}{n}.
\end{align}

Similarly ,
\begin{equation}\label{three2}
\left|\bar{V}^{(i)}(n) - \bar{C}^{(i)}(n) \right|<  D \dfrac{\gamma(n)}{n}.
\end{equation}

Then,
\begin{align*}
& \dfrac{b_n}{a_n-1} \ds \sum_{k=0}^{a_n-1} \left|\left( \bar{V}^{(i)}(n) - \bar{C}^{(i)}(n) \right)\left( \bar{V}^{(j)}(n) - \bar{C}^{(j)}(n) \right) \right| \\
& < \dfrac{b_n}{a_n-1} a_n D^2 \dfrac{\gamma(n)^2}{n^2}\\
& = D^2\dfrac{a_n}{a_n-1} \dfrac{b_n}{n} \dfrac{b_n}{n} \dfrac{\gamma(n)^2}{b_n}\\
& < D^2\dfrac{a_n}{a_n-1} \dfrac{b_n}{n} \dfrac{b_n}{n} \dfrac{\gamma(n)^2 \log n}{b_n}\\
& \to 0 \text{ w.p 1 as } n \to \infty \text{ by Condition \prettyref{Scond:bn_conditions}\prettyref{Scond:bn_to_n}}.
\end{align*}

\item By \eqref{two1} and \eqref{three2}, we have
\begin{align*}
&\dfrac{b_n}{a_n-1} \ds \sum_{k=0}^{a_n-1} \left| \left(\bar{V}^{(i)}_k - \bar{C}^{(i)}_k \right)  \left(\bar{V}^{(j)}(n) - \bar{C}^{(j)}(n) \right) \right| \\
& \leq \dfrac{b_n}{a_n-1} a_n \left( \dfrac{2D}{b_n} \gamma(n) \right) \left( \dfrac{D}{n} \gamma(n)  \right)\\
& < 2 D^2 \dfrac{a_n}{a_n-1} \dfrac{b_n}{n} \dfrac{\gamma(n)^2 \log n}{b_n}\\
& \to 0 \text{ w.p. 1 as }  n \to \infty \text{ by Condition \prettyref{Scond:bn_conditions}\prettyref{Scond:bn_to_n}}.
\end{align*}

\item By (\ref{two2}) and (\ref{three1}), we have
\begin{align*}
& \dfrac{b_n}{a_n-1} \ds \sum_{k=0}^{a_n-1} \left| \left(\bar{V}^{(i)}(n) - \bar{C}^{(i)}(n) \right)\left(\bar{V}^{(j)}_k - \bar{C}^{(j)}_k \right)   \right| \\
& \leq \dfrac{b_n}{a_n-1} a_n \left( \dfrac{2D}{b_n} \gamma(n) \right) \left( \dfrac{D}{n} \gamma(n) \right)\\
& < 2 D^2 \dfrac{a_n}{a_n-1} \dfrac{b_n}{n} \dfrac{\gamma(n)^2 \log n}{b_n}\\
& \to 0 \text{ w.p. 1 as }  n \to \infty \text{ by Condition \prettyref{Scond:bn_conditions}\prettyref{Scond:bn_to_n}}.
\end{align*}

\item By Corollary \ref{cor:seconditer} and (\ref{two1})
\begin{align*}
& \dfrac{b_n}{a_n-1} \ds \sum_{k=0}^{a_n-1} \left| \left(\bar{V}^{(i)}_k - \bar{C}^{(i)}_k  \right)\bar{C}^{(j)}_k \right| \\
& <  \dfrac{b_n}{a_n-1} a_n \left( \dfrac{2D}{b_n} \gamma(n) \right) \left( \sqrt{\dfrac{2\Sigma_{ii}}{b_n}} (1+ \epsilon) \left(\log \dfrac{n}{b_n} + \log \log n  \right)^{1/2} \right)\\
& < 2^{3/2}\Sigma_{ii}^{1/2} D (1+\epsilon)\dfrac{a_n}{a_n-1} \dfrac{\gamma(n)}{\sqrt{b_n}} (2 \log n)^{1/2} \\
& \to 0 \text{ w.p. 1 as }  n \to \infty \text{ by Condition \prettyref{Scond:bn_conditions}\prettyref{Scond:bn_to_n}}.
\end{align*}

\item By Corollary \ref{cor:seconditer} and (\ref{two2})
\begin{align*}
& \dfrac{b_n}{a_n-1} \ds \sum_{k=0}^{a_n-1} \left| \left(\bar{V}^{(j)}_k - \bar{C}^{(j)}_k  \right)\bar{C}^{(i)}_k \right| \\
& <  \dfrac{b_n}{a_n-1} a_n \left( \dfrac{2D}{b_n} \gamma(n) \right) \left( \sqrt{\dfrac{2\Sigma_{ii}}{b_n}} (1+ \epsilon) \left(\log \dfrac{n}{b_n} + \log \log n  \right)^{1/2} \right)\\
& < 2^{3/2}\Sigma_{ii}^{1/2} D (1+\epsilon)\dfrac{a_n}{a_n-1} \dfrac{\gamma(n)}{\sqrt{b_n}} (2 \log n)^{1/2} \\
& \to 0 \text{ w.p. 1 as }  n \to \infty \text{ by Condition \prettyref{Scond:bn_conditions}\prettyref{Scond:bn_to_n}}.
\end{align*}

\item  By Corollary \ref{cor:iter} and (\ref{two1})
\begin{align*}
& \dfrac{b_n}{a_n-1} \ds \sum_{k=0}^{a_n-1} \left|\left(\bar{V}^{(i)}_k - \bar{C}^{(i)}_k  \right)\bar{C}^{(j)}(n) \right| \\
& <  \dfrac{b_n}{a_n-1} a_n \left( \dfrac{2D}{b_n} \gamma(n) \right) \left( \dfrac{1}{n} (1+ \epsilon) \left( 2 \Sigma_{ii} n\log \log n \right)^{1/2} \right)\\
& < 2^{3/2} D \sqrt{\Sigma_{ii}} (1+ \epsilon) \dfrac{a_n}{a_n-1} \dfrac{ \gamma(n)(\log n)^{1/2}}{n^{1/2}}\\
& = 2^{3/2} D \sqrt{\Sigma_{ii}} (1+ \epsilon) \dfrac{a_n}{a_n-1} \left(\dfrac{b_n}{n}\right)^{1/2} \dfrac{\gamma(n) (\log n)^{1/2}}{b_n^{1/2}}\\
& \to 0 \text{ w.p. 1 as }  n \to \infty \text{ by Condition \prettyref{Scond:bn_conditions}\prettyref{Scond:bn_to_n}}.
\end{align*}

\item  By Corollary \ref{cor:iter} and (\ref{two2})
\begin{align*}
& \dfrac{b_n}{a_n-1} \ds \sum_{k=0}^{a_n-1} \left|\left(\bar{V}^{(j)}_k - \bar{C}^{(j)}_k  \right)\bar{C}^{(i)}(n) \right| \\
& < \dfrac{b_n}{a_n-1} a_n \left( \dfrac{2D}{b_n} \gamma(n) \right) \left( \dfrac{1}{n} (1+ \epsilon) \left( 2 \Sigma_{ii} n\log \log n \right)^{1/2} \right)\\
& < 2^{3/2} D \sqrt{\Sigma_{ii}} (1+ \epsilon) \dfrac{a_n}{a_n-1} \dfrac{ \gamma(n)(\log n)^{1/2}}{n^{1/2}}\\
& = 2^{3/2} D \sqrt{\Sigma_{ii}} (1+ \epsilon) \dfrac{a_n}{a_n-1} \left(\dfrac{b_n}{n}\right)^{1/2} \dfrac{\gamma(n) (\log n)^{1/2}}{b_n^{1/2}}\\
& \to 0 \text{ w.p. 1 as }  n \to \infty \text{ by Condition \prettyref{Scond:bn_conditions}\prettyref{Scond:bn_to_n}}.
\end{align*}

\item By (\ref{three1}) and Corollary \ref{cor:seconditer}
\begin{align*}
& \dfrac{b_n}{a_n-1}\ds \sum_{k=0}^{a_n-1} \left|\left(\bar{V}^{(i)}(n) - \bar{C}^{(i)}(n) \right) \bar{C}^{(j)}_k \right| \\
& <  \dfrac{b_n}{a_n-1} a_n \left(\dfrac{D}{n} \gamma(n) \right) \left(\sqrt{\dfrac{2\Sigma_{ii}}{b_n}} (1+\epsilon) \left(\log \dfrac{n}{b_n} + \log \log n\right)^{1/2} \right)\\
& < \sqrt{2\Sigma_{ii}} D (1+ \epsilon) \dfrac{a_n}{a_n-1} \dfrac{b_n}{n} \gamma(n) \left(\dfrac{2}{b_n} \log n   \right)^{1/2}\\
& = 2^{3/2} \Sigma_{ii}^{1/2} D (1+\epsilon) \dfrac{a_n}{a_n - 1} \dfrac{b_n}{n} \dfrac{\gamma(n) (\log n)^{1/2}}{b_n^{1/2}}\\
& \to 0 \text{ w.p. 1 as }  n \to \infty \text{ by Condition \prettyref{Scond:bn_conditions}\prettyref{Scond:bn_to_n}}.
\end{align*}

\item By (\ref{three2}) and Corollary \ref{cor:seconditer}
\begin{align*}
& \dfrac{b_n}{a_n-1}\ds \sum_{k=0}^{a_n-1} \left|\left(\bar{V}^{(j)}(n) - \bar{C}^{(j)}(n) \right) \bar{C}^{(i)}_k \right| \\
& <  \dfrac{b_n}{a_n-1} a_n \left(D \dfrac{\gamma(n)}{n} \right) \left(\sqrt{\dfrac{2\Sigma_{ii}}{b_n}} (1+\epsilon) \left(\log \dfrac{n}{b_n} + \log \log n\right)^{1/2} \right)\\
& < \sqrt{2\Sigma_{ii}} D (1+ \epsilon) \dfrac{a_n}{a_n-1} \dfrac{b_n}{n} \gamma(n) \left(\dfrac{2}{b_n} \log n   \right)^{1/2}\\
& = 2^{3/2} \Sigma_{ii}^{1/2} D (1+\epsilon) \dfrac{a_n}{a_n - 1} \dfrac{b_n}{n} \dfrac{\gamma(n) (\log n)^{1/2}}{b_n^{1/2}}\\
& \to 0 \text{ w.p. 1 as }  n \to \infty \text{ by Condition \prettyref{Scond:bn_conditions}\prettyref{Scond:bn_to_n}}.
\end{align*}

\item By (\ref{three1}) and Corollary \ref{cor:iter}
\begin{align*}
& \dfrac{b_n}{a_n-1} \ds \sum_{k=0}^{a_n-1}\left|\left(\bar{V}^{(i)}(n) - \bar{C}^{(i)}(n)\right) \bar{C}^{(j)}(n) \right| \\
& < \dfrac{b_n}{a_n-1} a_n \left( \dfrac{D}{n} \gamma(n) \right) \left( \dfrac{1}{n}(1+ \epsilon) (2\Sigma_{ii} n \log \log n)^{1/2} \right)\\
& < \sqrt{2\Sigma_{ii}} D (1+ \epsilon) \dfrac{a_n}{a_n-1} \dfrac{b_n}{n} \dfrac{\gamma(n) (\log n)^{1/2}}{n^{1/2}} \\
& \to 0 \text{ w.p. 1 as }  n \to \infty \text{ by Condition \prettyref{Scond:bn_conditions}\prettyref{Scond:bn_to_n}}.
\end{align*}

\item By (\ref{three2}) and Corollary \ref{cor:iter}
\begin{align*}
& \dfrac{b_n}{a_n-1} \ds \sum_{k=0}^{a_n-1}\left|\left(\bar{V}^{(j)}(n) - \bar{C}^{(j)}(n)\right) \bar{C}^{(i)}(n) \right| \\
& < \dfrac{b_n}{a_n-1} a_n \left( D \dfrac{\gamma(n)}{n}  \right) \left( \dfrac{1}{n}(1+ \epsilon) (2\Sigma_{ii} n \log \log n)^{1/2} \right)\\
& < \sqrt{2\Sigma_{ii}} D (1+ \epsilon) \dfrac{a_n}{a_n-1} \dfrac{b_n}{n} \dfrac{\gamma(n) (\log n)^{1/2}}{n^{1/2}} \\
& \to 0 \text{ w.p. 1 as }  n \to \infty \text{ by Condition \prettyref{Scond:bn_conditions}\prettyref{Scond:bn_to_n}}.
\end{align*}
\end{enumerate}
Thus, each of the 13 terms tends to 0 with probability 1 as $n \to \infty$, giving that $\Sigma_{n,ij} \to \Sigma_{ij}$ w.p. 1 as $n \to \infty.$

\section{Proof of Theorem~\prettyref{Sthm:geom_erg_logistic}}
\label{subsec:geom_erg_logistic}

Without loss of generality, we assume $\tau^2 = 1$. The posterior distribution for this Bayesian logistic regression model is,
\begin{align*}
\label{eq:post}
  f(\beta | y, x) & \propto f(\beta) \prod_{i=1}^{K} f(y_i|x_i, \beta)\\
  & \propto e^{-\frac{1}{2} \beta^T\beta} \prod_{i=1}^{K} \left( \dfrac{1}{1 + e^{-x_i \beta}} \right)^{y_i} \left(\dfrac{e^{-x_i \beta}}{1 + e^{-x_i \beta}}   \right)^{1-y_i}. \numberthis
\end{align*}
For simpler notation we will use $f(\beta)$ to denote the posterior density. Note that the posterior has a moment generating function.

Consider a random walk Metropolis-Hastings algorithm with a multivariate normal proposal distribution to sample from the posterior $f(\beta)$. We will use the following result to establish geometric ergodicity of this Markov chain.

\begin{theorem}[\cite{jarn:hans:2000}]
\label{jarn_and_hans}

Let $m(\beta) = \nabla f(\beta) / \|\nabla f(\beta)\|$ and also let $n(\beta) = \beta / \|\beta\|$. Suppose $f$ on $\mathbb{R}^p$ is super-exponential in that it is positive and has continuous first derivatives such that
\begin{equation} \label{eq:jh1}
\lim_{\|\beta\| \to \infty} n(\beta) \cdot \nabla \log f(\beta) = - \infty \, . 
\end{equation}
 In addition let the proposal distribution be bounded away from 0 in some region around zero. If
 \begin{equation} \label{eq:jh2}
\limsup_{\|\beta\| \to \infty} n(\beta) \cdot m(\beta) < 0\, ,
\end{equation}
 then the random walk Metropolis-Hastings algorithm is geometrically ergodic.  \end{theorem}

\begin{proof}[Proof of Theorem~\prettyref{Sthm:geom_erg_logistic}]
Note that the multivariate normal proposal distribution $q$ is indeed bounded away from zero in some region around zero. We will first show that $f$ is super-exponential.  It is easy to see that $f$ has continuous first derivatives and is positive.  Next we need to establish \eqref{eq:jh1}.  From \eqref{eq:post} we see that
\begin{align*}
& \log f(\beta) \\
& = \text{const} -\dfrac{1}{2} \beta^T \beta - \ds \sum_{i=1}^{K} y_i \log(1 + e^{-x_i \beta}) - \ds \sum_{i=1}^{K} (1-y_i) x_i \beta - \ds \sum_{i=1}^{K} (1-y_i) \log (1 + e^{-x_i \beta})\\
& = \text{const} -\dfrac{1}{2} \beta^T \beta - \ds \sum_{i=1}^{K} \log(1 + e^{-x_i \beta}) - \ds \sum_{i=1}^{K} (1-y_i) x_i \beta\\
& = \text{const} -\dfrac{1}{2} \ds \sum_{j = 1}^{p} \beta_j^2 - \ds \sum_{i=1}^{K} \log(1 + e^{-\sum_{j=1}^{p} x_{ij}\beta_j}) - \ds \sum_{i=1}^{K} (1-y_i)\ds \sum_{j=1}^{p} x_{ij}\beta_j\, .
\end{align*}
For $l=1,\ldots, p$
\begin{align*}
\dfrac{\partial \log f(\beta)}{\partial \beta_l} & = -\beta_l + \ds \sum_{i=1}^{K} \dfrac{x_{il} e^{-x_i \beta}}{1 + e^{-x_i \beta}} - \ds \sum_{i=1}^{K} (1 - y_i) x_{il}
\end{align*}
and
\begin{align*}
\beta \cdot \nabla \log f(\beta) & = \ds \sum_{j=1}^{p}\left[ - \beta_j^2 + \ds \sum_{i=1}^{K} x_{ij}\beta_j \dfrac{e^{-x_i \beta}}{1 + e^{-x_i \beta}} - \ds \sum_{i=1}^{K} (1-y_i) x_{ij} \beta_j  \right]\\
& = -\|\beta\|^2 + \ds \sum_{i=1}^{K} x_i \beta \dfrac{e^{-x_i \beta}}{1 + e^{-x_i \beta}} - \ds \sum_{i=1}^{K} (1 - y_i) x_i \beta \, .
\end{align*}
Hence
\begin{align*}
\dfrac{\beta}{\|\beta\|}\cdot \nabla \log f(\beta) & = - \|\beta\| + \ds \sum_{i=1}^{K} \dfrac{x_i \beta}{\|\beta\|} \dfrac{e^{-x_i \beta}}{1 + e^{-x_i\beta}} - \ds \sum_{i=1}^{K} (1-y_i)\dfrac{x_i \beta}{\|\beta\|}.
\end{align*}
Taking the limit with $\|\beta\| \to \infty$ we obtain
\begin{align*}
\lim_{\|\beta\| \to \infty} \dfrac{\beta}{\|\beta\|} \cdot \nabla \log f(\beta) &  = - \lim_{\|\beta\| \to \infty} \|\beta\| + \lim_{\|\beta\| \to \infty} \ds \sum_{i=1}^{K} \dfrac{x_i \beta}{\|\beta\|} \dfrac{e^{-x_i \beta}}{1 + e^{-x_i\beta}}\\
& \quad - \lim_{\|\beta\| \to \infty} \ds \sum_{i=1}^{K} (1-y_i)\dfrac{x_i \beta}{\|\beta\|}. \numberthis \label{eq:lim_eq} 
\end{align*}
By the Cauchy-Schwarz inequality we can bound the second term
\begin{align*}
\lim_{\|\beta\| \to \infty} \ds \sum_{i=1}^{K} \dfrac{x_i \beta}{\|\beta\|} \dfrac{e^{-x_i \beta}}{1 + e^{-x_i\beta}} & \leq \lim_{\|\beta\| \to \infty} \ds \sum_{i=1}^{K} \dfrac{|x_i| \|\beta\|}{\|\beta\|} \dfrac{e^{-x_i \beta}}{1 + e^{-x_i\beta}} \leq \ds \sum_{i=1}^{K} |x_{i}|. \numberthis \label{eq:2ndterm} 
\end{align*}
For the third term we obtain
\begin{align*}
\lim_{\|\beta\| \to \infty} \ds \sum_{i=1}^{K} (1-y_i)\dfrac{x_i \beta}{\|\beta\|} & = \lim_{\|\beta\| \to \infty} \ds \sum_{i=1}^{K} (1-y_i)\dfrac{\sum_{j=1}^{p} x_{ij} \beta_j}{\|\beta\|} \\
& = \ds \sum_{i=1}^{K}(1 - y_i) \ds \sum_{j=1}^{p} \lim_{\|\beta\| \to \infty} \dfrac{x_{ij} \beta_j}{\|\beta\|} \\
& \geq \ds \sum_{i=1}^{K}(1 - y_i) \ds \sum_{j=1}^{p} \lim_{\|\beta\| \to \infty} \dfrac{-|x_{ij}| |\beta_j|}{\|\beta\|}\\
& \geq - \ds \sum_{i=1}^{K}(1 - y_i) \ds \sum_{j=1}^{p} \lim_{\|\beta\| \to \infty} |x_{ij}| \quad \quad \text{Since $|\beta_j| \leq \|\beta\|$}\\
& = - \ds \sum_{i=1}^{K}(1- y_i) \|x_{i}\|_1 \numberthis \label{eq:3rdterm}.
\end{align*}
Using \eqref{eq:2ndterm} and \eqref{eq:3rdterm} in \eqref{eq:lim_eq}.
\begin{align*}
\lim_{\|\beta\| \to \infty} \dfrac{\beta}{\|\beta\|} \cdot \nabla \log f(\beta) & \leq - \lim_{\|\beta\| \to \infty} \|\beta\| + \ds \sum_{i=1}^{K} |x_{i}| + \ds \sum_{i=1}^{K}(1- y_i) \|x_{i}\|_1 = - \infty\, .
\end{align*}
Next we need to establish~\eqref{eq:jh2}.  Notice that 
\begin{align*}
f(\beta)& \propto \exp \left[ - \dfrac{1}{2} \ds \sum_{j=1}^{p} \beta_j^2 - \ds \sum_{i=1}^{K} (1-y_i) \sum_{j=1}^{p} x_{ij}\beta_j - \ds \sum_{i=1}^{K} \log(1 + e^{- \sum_{j=1}^{p}x_{ij}\beta_j})  \right] := e^{C(\beta)} 
\end{align*}
and hence for $l=1,\ldots,p$
\begin{align*}
\dfrac{\partial f(\beta)}{\partial \beta_l} & = e^{C(\beta)} \left[ - \beta_l - \ds \sum_{i=1}^{K} (1-y_i) x_{il} - \ds \sum_{i=1}^{K} \dfrac{-x_{il} e^{-x_i\beta}}{1 + e^{-x_i\beta} } \right].
\end{align*}
In order to show the result, we will need evaluate
\[ \lim_{\|\beta\| \to \infty}\dfrac{  e^{C(\beta)} \|\beta\|}{\|\nabla f(\beta) \|}. \]
To this end, we will first show that
\[\lim_{\|\beta \| \to \infty} \dfrac{\| \nabla f(\beta) \|^2}{e^{2C(\beta)} \|\beta \|^2} = 1. \]
We calculate that
\begin{align*}
&\|\nabla f(\beta) \|^2\\
& = e^{2C(\beta)} \ds \sum_{j=1}^{p} \left[ - \beta_j - \ds \sum_{i=1}^{K} (1-y_i) x_{ij} + \ds \sum_{i=1}^{K} \dfrac{x_{ij} e^{-x_i\beta}}{1 + e^{-x_i\beta} } \right]^2  \\
& = e^{2C(\beta)}  \ds \sum_{j=1}^{p}  \left[ \left(   \ds \sum_{i=1}^{K} \dfrac{x_{ij} e^{-x_i\beta}}{1 + e^{-x_i\beta} } \right)^2 + \beta_j^2 + \left( \ds \sum_{i=1}^{K} (1-y_i) x_{ij} \right)^2 + 2 \ds \sum_{i=1}^{K} (1 - y_i) x_{ij} \beta_j \right.\\
& \quad  \left. - 2\left( \beta_j + \ds \sum_{i=1}^{K} (1 - y_i) x_{ij}  \right) \left(\ds \sum_{i=1}^{K} x_{ij} \dfrac{e^{-x_i \beta}}{1 + e^{-x_i \beta}}   \right) \right]\\
& = e^{2C(\beta)} \left[ \ds \sum_{j=1}^{p} \left(   \ds \sum_{i=1}^{K} \dfrac{x_{ij} e^{-x_i\beta}}{1 + e^{-x_i\beta} } \right)^2 + \|\beta\|^2 + \ds \sum_{j=1}^{p} \left( \ds \sum_{i=1}^{K}(1-y_i)x_{ij}  \right)^2 + 2\ds \sum_{i=1}^{K} (1-y_i)x_i \beta \right.\\
& \quad \left. -2\ds \sum_{i=1}^{K} x_i \beta \dfrac{e^{-x_i\beta}}{1 + e^{-x_i\beta} } - 2 \ds \sum_{j=1}^{p} \left( \ds \sum_{i=1}^{K}(1-y_i)x_i \right) \left( \ds \sum_{i=1}^{K} x_{ij} \dfrac{e^{-x_i\beta}}{1 + e^{-x_i\beta}} \right) \right]\\
& = e^{2C(\beta)} \|\beta\|^2 \left[ \dfrac{1}{\|\beta\|^2} \ds \sum_{j=1}^{p} \left(   \ds \sum_{i=1}^{K} \dfrac{x_{ij} e^{-x_i\beta}}{1 + e^{-x_i\beta} } \right)^2 + 1 + \dfrac{1}{\|\beta\|^2} \ds \sum_{j=1}^{p} \left( \ds \sum_{i=1}^{K}(1-y_i)x_{ij}  \right)^2  \right.\\
& \quad \left.+ 2 \ds \sum_{i=1}^{K} (1-y_i) \dfrac{x_i \beta}{\|\beta\|^2}  -2\ds \sum_{i=1}^{K} \dfrac{x_i \beta}{\|\beta\|^2} \dfrac{e^{-x_i\beta}}{1 + e^{-x_i\beta} }  \right.\\
& \quad \left. - 2 \dfrac{1}{\|\beta\|^2} \ds \sum_{j=1}^{p} \left( \ds \sum_{i=1}^{K}(1-y_i)x_i \right) \left( \ds \sum_{i=1}^{K} x_{ij} \dfrac{e^{-x_i\beta}}{1 + e^{-x_i\beta}} \right) \right]
\end{align*}
and
\begin{align*}
& \dfrac{\|\nabla f(\beta)\|^2}{e^{2C(\beta)} \|\beta\|^2} \\
& = \dfrac{1}{\|\beta \|}\left[ \dfrac{1}{\|\beta\|} \ds \sum_{j=1}^{p} \left(   \ds \sum_{i=1}^{K} \dfrac{x_{ij} e^{-x_i\beta}}{1 + e^{-x_i\beta} } \right)^2  + \dfrac{1}{\|\beta\|} \ds \sum_{j=1}^{p} \left( \ds \sum_{i=1}^{K}(1-y_i)x_{ij}  \right)^2 + 2 \ds \sum_{i=1}^{K} (1-y_i) \dfrac{x_i \beta}{\|\beta\|}  \right.\\
& \quad \left. -2\ds \sum_{i=1}^{K} \dfrac{x_i \beta}{\|\beta\|} \dfrac{e^{-x_i\beta}}{1 + e^{-x_i\beta} } - 2 \dfrac{1}{\|\beta\|} \ds \sum_{j=1}^{p} \left( \ds \sum_{i=1}^{K}(1-y_i)x_i \right) \left( \ds \sum_{i=1}^{K} x_{ij} \dfrac{e^{-x_i\beta}}{1 + e^{-x_i\beta}} \right) \right] + 1. \numberthis \label{eq:nablaf}  
\end{align*}
Since $\lim_{\|\beta\| \to \infty} \|\beta\|^{-1} \to 0$, it is left to show that the term in the square brackets is bounded in the limit. Since $y_i$ and $x_{ij}$ are independent of $\beta$, and $e^{-x_i\beta}/(1 + e^{-x_i\beta})$ bounded below by 0 and above by 1, it is only required to show that the third and fourth terms in the square brackets remain bounded in the limit. From \eqref{eq:3rdterm} and the Cauchy-Schwarz inequality
\begin{align*}
- \ds \sum_{i=1}^{K}(1- y_i) ||x_{i}||_1  \leq  2 \lim_{\|\beta\| \to \infty} \ds \sum_{i=1}^{K} (1-y_i) \dfrac{x_i \beta}{\|\beta\|} \leq 2\ds \sum_{i=1}^{K} (1- y_i) |x_i| \, .
\end{align*}
In addition, 
\begin{align*}
  \ds \sum_{i=1}^{K} \dfrac{x_i \beta}{\|\beta\|^2} \dfrac{e^{-x_i\beta}}{1 + e^{-x_i\beta} } & \geq - \ds \sum_{i=1}^{K} \sum_{j=1}^{p} \dfrac{|x_{ij}||\beta_j|}{\|\beta\|} \dfrac{e^{-x_i\beta}}{1 + e^{-x_i\beta} }\\
  & \geq - \ds \sum_{i=1}^{K} \sum_{j=1}^{p} \dfrac{|x_{ij}||\beta_j|}{\|\beta\|}\\
  & \geq - \ds \sum_{i=1}^{K} ||x_i||_1.
\end{align*}
Thus, by the above result and \eqref{eq:2ndterm},
\begin{align*}
- \ds \sum_{i=1}^{K} ||x||_1  \leq \lim_{\|\beta\| \to \infty} \ds \sum_{i=1}^{K} \dfrac{x_i \beta}{\|\beta\|} \dfrac{e^{-x_i\beta}}{1 + e^{-x_i\beta} } \leq \ds \sum_{i=1}^{K} |x_{i}|.
\end{align*}
Using these results in \eqref{eq:nablaf},
\begin{align*}
& \lim_{\|\beta\| \to \infty} \dfrac{\|\nabla f(\beta)\|^2 }{e^{2C(\beta)} \|\beta\|^2} \\
& = 1 + \lim_{\|\beta\| \to \infty} \dfrac{1}{\|\beta \|}\left[ \dfrac{1}{\|\beta\|} \ds \sum_{j=1}^{p} \left(   \ds \sum_{i=1}^{K} \dfrac{x_{ij} e^{-x_i\beta}}{1 + e^{-x_i\beta} } \right)^2  + \dfrac{1}{\|\beta\|} \ds \sum_{j=1}^{p} \left( \ds \sum_{i=1}^{K}(1-y_i)x_{ij}  \right)^2\right.\\
& \quad + 2 \ds \sum_{i=1}^{K} (1-y_i) \dfrac{x_i \beta}{\|\beta\|}   -2\ds \sum_{i=1}^{K} \dfrac{x_i \beta}{\|\beta\|} \dfrac{e^{-x_i\beta}}{1 + e^{-x_i\beta} }\\
& \quad  \left. - 2 \dfrac{1}{\|\beta\|} \ds \sum_{j=1}^{p} \left( \ds \sum_{i=1}^{K}(1-y_i)x_i \right) \left( \ds \sum_{i=1}^{K} x_{ij} \dfrac{e^{-x_i\beta}}{1 + e^{-x_i\beta}} \right) \right]\\
& = 1\, .
\end{align*}
Next observe that
\begin{align*}
\beta \cdot \nabla f(\beta) & = e^{C(\beta)}\ds \sum_{j=1}^{p} \left[-\beta_j^2 - \ds \sum_{i=1}^{K} (1 - y_i)x_{ij} \beta_j + \ds \sum_{i=1}^{K} x_{ij} \beta_j \dfrac{e^{-x_i \beta}}{1 + e^{-x_i \beta}}   \right]\\
& = e^{C(\beta)} \left[ -\|\beta\|^2 - \ds \sum_{i=1}^{K} (1 - y_i)x_i \beta + \ds \sum_{i=1}^{K} x_i \beta \dfrac{e^{-x_i \beta}}{1 + e^{-x_i \beta}}  \right]\,.
\end{align*}
and hence
\begin{align*}
\dfrac{\beta}{\|\beta\|}  \dfrac{\nabla f(\beta)}{\| \nabla f(\beta)\|} &= \dfrac{e^{C(\beta)}}{\|\nabla f(\beta) \|} \left[ -\|\beta\| - \ds \sum_{i=1}^{K} (1-y_i) \dfrac{x_i \beta}{\|\beta\|} + \ds \sum_{i=1}^{K} \dfrac{x_i \beta}{\|\beta\|} \dfrac{e^{-x_i \beta}}{1 + e^{-x_i \beta}}    \right]\, .
\end{align*}
We conclude that
\begin{align*}
\lim_{\|\beta\| \to \infty} \dfrac{\beta}{\|\beta\|}  \dfrac{\nabla f(\beta)}{\| \nabla f(\beta)\|} & = \lim_{\|\beta\| \to \infty} \dfrac{e^{C(\beta)} \|\beta\|}{\|\nabla f(\beta) \|} \left[ -1 - \ds \sum_{i=1}^{K} (1-y_i) \dfrac{x_i \beta}{\|\beta\|^2} + \ds \sum_{i=1}^{K} \dfrac{x_i \beta}{\|\beta\|^2} \dfrac{e^{-x_i \beta}}{1 + e^{-x_i \beta}}    \right]\\
& = -1 
\end{align*}
which establishes \eqref{eq:jh2}.
\end{proof}

\section{Univariate Termination Rules}
\label{sec:uni_term_rules}
In this section we formally present the univariate termination rules we implement in Section~\prettyref{Ssec:examples} of the main document. Recall that for the $i$th component of $\theta$, $\theta_{n,i}$ is the Monte Carlo estimate for $\theta_i$, and $\sigma_i^2$ is the asymptotic variance in the univariate CLT. Let $\sigma_{n,i}^2$ be the univariate batch means estimator of $\sigma_i^2$ and let $\lambda_{n,i}^2$ be the sample variance for the $i$th component.

Common practice is to terminate simulation when \textit{all} components satisfy a termination rule. We focus on the relative standard deviation fixed-width sequential stopping rules. Due to multiple testing, a Bonferroni correction is often used. We will refer to the uncorrected method as the uBM method and the corrected method as the uBM-Bonferroni. To create $100(1-\alpha)\%$ univariate confidence intervals, the relative standard deviation fixed-width rule terminates at the random time,
\begin{equation*}
  \inf\left\{n \geq 0 : \dfrac{1}{\lambda_{n,i}} \left(2t_*\dfrac{\sigma_{n,i}} {\sqrt{n}} + \epsilon_i\lambda_{n,i} I(n < n^*) +\dfrac{1}{n} \right)  \leq \epsilon_i  \quad \text{for all } i = 1, \dots, p\right\}\;,
\end{equation*}
where for uncorrected intervals $t_* =  t_{1-\alpha/2, a_n-1}$ and for Bonferroni corrected intervals $t_* =  t_{1-\alpha/2p, a_n-1}$. See \cite{fleg:gong:2015} for more details.

\begin{remark}
If it is unclear which features of $F$ are of interest \textit{a priori}, i.e., $g$ is not known before simulation, one can implement Scheffe's simultaneous confidence intervals. This method can also be used to create univariate confidence intervals for arbitrary contrasts \textit{a posteriori}. Scheffe's simultaneous intervals are boxes on the coordinate axes that bound the confidence ellipsoids. For all  $a^T\theta$, where $a \in \mathbb{R}^p$,
\[ a^T \theta_n \pm \sqrt{a^T \Sigma_n a  ~~ \dfrac{p (a_n - 1)}{(a_n - p)} F_{1-\alpha, p, a_n-p} } \]
will have simultaneous coverage probability $1 - \alpha$. Since we generally are not interested in all possible linear combinations, Scheffe's intervals are often too conservative. In fact, if the number of  confidence intervals is less than $p$, Scheffe's intervals are more conservative than Bonferroni corrected intervals.
\end{remark}

\section{Computational Cost}
The multivariate termination rule is naturally more expensive than the univariate termination rule, but in our experience, the time to check the termination criterion is insignificant compared to the time it takes for obtaining more samples. Also note that there are cases when it is almost impossible to meet the  termination criterion in machine time using univariate methods. This was demonstrated for the Bayesian dynamic spatial-temporal example.

We present computational time to termination for the Bayesian logistic and the VAR(1) models. For the VAR(1) model, we keep the eigen-structure for $\Phi$ as the same as in Section~\prettyref{Ssec:var} of the main document, but let $p$ = 50. Samples from both of the MCMC processes are obtained fairly quickly due to the inexpensive structure. Calculating the multivariate estimators for the VAR(1) model will take more time since we are estimating a $50 \times 50$ matrix. However, since one component mixes slowly, termination  by univariate methods is delayed.

In Table \ref{tab:logistic_time} we present mean computational time to termination using mBM and uBM-Bonferroni over 100 replications for two values of the precision $\epsilon$ (we make 90\% confidence regions). Time to termination is significantly lower for mBM methods, although the difference might not be practically significant. It is important to note that in the Bayesian logistic regression sampler, all the components seem to be mixing equally, so termination is not delayed due to that reason.

\begin{table}[h]
\footnotesize 
  \caption{\footnotesize  \label{tab:logistic_time}Bayesian Logistic: Computation time in seconds for termination using multivariate and univariate techniques. $ b_n  = \lfloor n^{1/2}\rfloor$.}
\begin{center}

  \begin{tabular}{|l|cc|}
  \hline
 & mBM & uBM-Bonferroni  \\ \hline
$\epsilon = .02$  &    42.16 {\tiny (0.308)}  &  43.30  {\tiny (0.327)}\\
$\epsilon = .01$  &   161.83 {\tiny (0.708)}  &  165.18  {\tiny (0.731)}\\
  \hline
  \end{tabular}
  \end{center}
\end{table}

In Table \ref{tab:var_time} we present mean computational time to termination for the VAR(1) model with $p = 50$ over 100 replications. Here, the Bonferroni correction and one slow mixing component both lead to delayed univariate termination, and large $p$ causes mBM calculation to be slower than univariate methods. However, we see significant gains in computational time using multivariate termination rules.

\begin{table}[h]
\footnotesize 
  \caption{ \footnotesize \label{tab:var_time}VAR(1): Computation time in seconds for termination using multivariate and univariate techniques. $ b_n  = \lfloor n^{1/3}\rfloor$.}

\begin{center}
  \begin{tabular}{|l|cc|}
  \hline
 & mBM & uBM-Bonferroni  \\ \hline
$\epsilon = .02$  &    40.56 {\tiny (0.044)}  &  58.11  {\tiny (0.061)}\\
$\epsilon = .01$  &   172.91 {\tiny (0.826)}  &  251.04  {\tiny (1.477)}\\
  \hline
  \end{tabular}
  \end{center}
\end{table}

Both these examples were in a way best case scenario for the univariate methods because obtaining samples is cheap for both examples. Even then we see multivariate termination methods require less computation time for the same level of precision $\epsilon$.

\section{Sensitivity to Batch Size}

We explore the finite sample properties of using large and small batch sizes for the VAR example. Table \ref{tab:batch_var}  has coverage probabilities over 1000 replications for different choices of tolerance level $\epsilon$ in the relative volume sequential stopping rule. Generally a smaller $\epsilon$ is chosen so as to ensure reasonable coverage probabilities.

The difference in the coverage probabilities is minimal for $\epsilon
\leq .05$ between the three batch sizes. Although, for larger
$\epsilon$, the batch size has a greater impact. However, for large batch sizes, the variability in the estimation of $\Sigma$ is larger since less number of batches are available for estimation. Figure \ref{fig:eigen} shows the estimated density of the largest estimated eigenvalue over 1000 replications.
\begin{figure}[h]
  \centering
  \includegraphics[width = 3in]{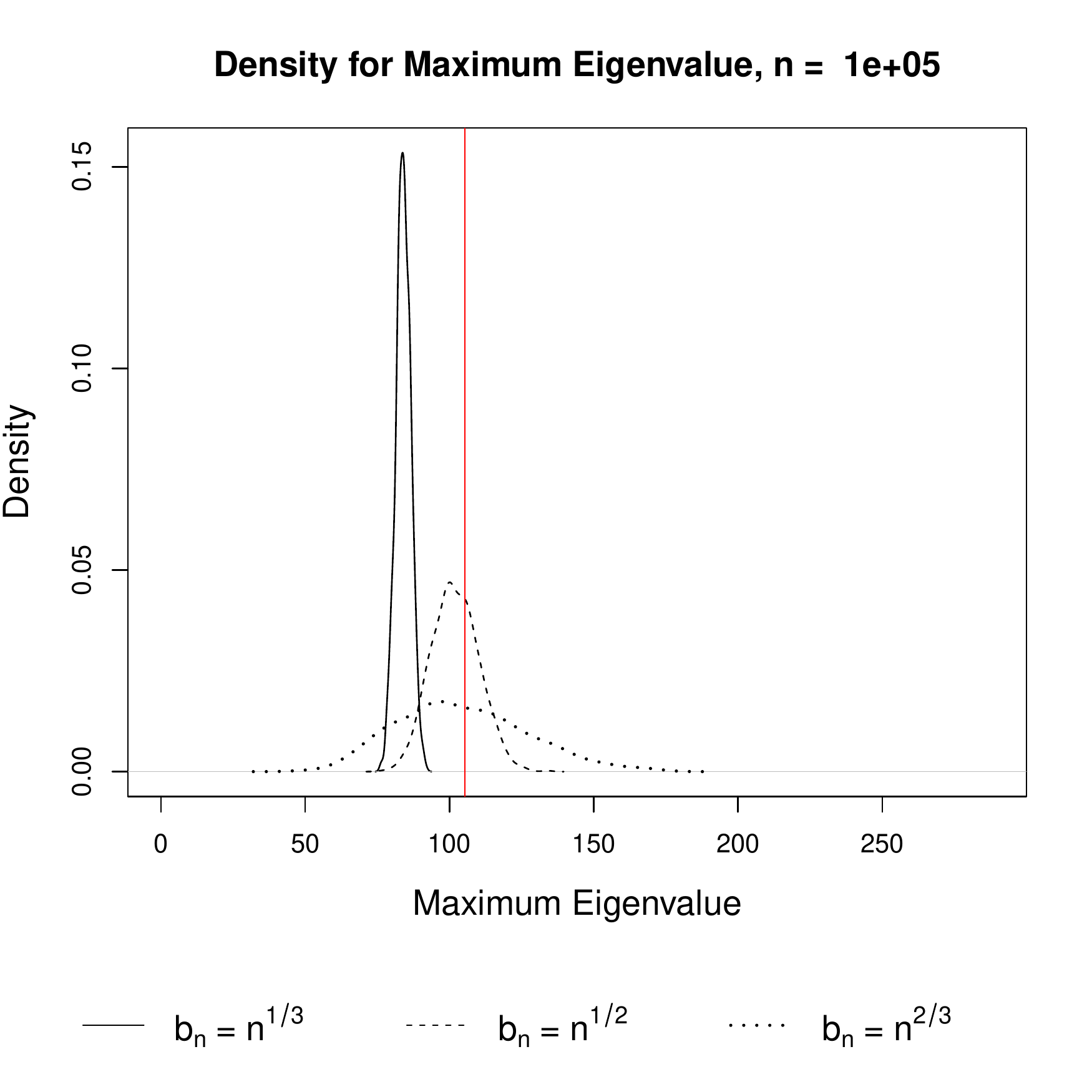}
  \caption{\footnotesize Var: Replications = 1000. Density of the largest estimated eigenvalue of $\Sigma$ using mBM estimator for three batch sizes.}
  \label{fig:eigen}
\end{figure}
Thus, for large batch sizes, the estimation is less biased, however, the variability in the estimates is considerably larger. The opposite phenomenon is witnessed for small batch sizes.  But since the Monte Carlo sample size is large enough, the batch size of $b_n = \lfloor n^{1/2} \rfloor$ has low bias and low variability.

\begin{table}[h]
\footnotesize
  \caption{\label{tab:batch_var}VAR Example. Coverage probability of 90\% confidence regions using mBM. Replications = 1000. $n^* = 1000$.}
\begin{center}
  \begin{tabular}{c|ccccc}
  \hline
 $b_n, \epsilon$  &  0.20   &  0.10   &   0.05   &   0.02   &   0.01 \\ \hline
$\lfloor n^{1/3} \rfloor$ &  0.812  \tiny{(0.0124)}     &    0.869  \tiny{(0.0107)}    &    0.886   \tiny{(0.0101)}   &     0.883   \tiny{(0.0102)}    &    0.900   \tiny{(0.0095)} \\
$\lfloor n^{1/2} \rfloor$  &  0.884  \tiny{(0.0101)}  &   0.898  \tiny{(0.0096)}  &  0.911   \tiny{(0.0090)}  &  0.894   \tiny{(0.0097)} &   0.909   \tiny{(0.0091)} \\
$\lfloor n^{2/3} \rfloor$  &  0.877  \tiny{(0.0104)}  &  0.905  \tiny{(0.0093)}  &  0.912   \tiny{(0.0090)}   & 0.894   \tiny{(0.0097)}  &  0.904   \tiny{(0.0093)} \\ \hline
  \end{tabular}
  \end{center}
\end{table}

\end{document}